\documentclass[letterpaper, 11pt]{article}

\usepackage[english]{babel}
\usepackage[margin=1in]{geometry}

\newcommand{\footremember}[2]{%
    \footnote{#2}
    \newcounter{#1}
    \setcounter{#1}{\value{footnote}}%
}
\newcommand{\footrecall}[1]{%
    \footnotemark[\value{#1}]%
}

\usepackage{graphicx} %
\usepackage[numbers,sort]{natbib}

\title{Near-Linear Runtime for a \\Classical Matrix Preconditioning Algorithm}
\author{Xufeng Cai\footremember{wisc}{Department of Computer Sciences, University of Wisconsin-Madison. XC (\href{mailto:xcai74@wisc.edu}{xcai74@wisc.edu}), JD (\href{mailto:jelena@cs.wisc.edu}{jelena@cs.wisc.edu}).} 
\and Jason M. Altschuler\footremember{upenn}{Department of Statistics and Data Science, University of Pennsylvania. \href{mailto:alts@upenn.edu}{alts@upenn.edu}.}
\and Jelena Diakonikolas\footrecall{wisc}}
\date{}

\usepackage{xcolor}
\definecolor{mydarkblue}{rgb}{0,0.08,0.45}
\usepackage[colorlinks,urlcolor=mydarkblue,citecolor=mydarkblue,linkcolor=mydarkblue]{hyperref}
\usepackage{times}
\usepackage{amsmath,amsfonts,amsthm,bm}

\usepackage{dsfont}
\usepackage{mathtools}

\usepackage{enumitem}
\usepackage{circledsteps}

\usepackage{color}

\usepackage{cleveref}

\usepackage{thm-restate}

\usepackage{algorithm}
\usepackage[noend]{algorithmic}

\usepackage{graphicx}
\usepackage{subcaption}

\newtheorem{lemma}{Lemma}
\newtheorem{corollary}{Corollary}

\newtheorem{proposition}{Proposition}

\newtheorem{fact}{Fact}

\newtheorem{problem}{Problem}
\newtheorem{implementation}{Implementation}

\def\1{\bm{1}}

\def\eps{{\epsilon}}

\def\ebar{\epsilon'}

\def\vone{{\bm{1}}}

\def\vc{{\bm{c}}}

\def\ve{{\bm{e}}}

\def\vr{{\bm{r}}}

\def\vu{{\bm{u}}}

\def\mA{{\bm{A}}}

\def\mD{{\bm{D}}}

\def\mM{{\bm{M}}}

\def\cO{{\mathcal{O}}}

\newcommand{\R}{\mathbb{R}}

\DeclareMathOperator*{\argmin}{arg\,min}

\usepackage[colorinlistoftodos]{todonotes}

\newcommand{\bl}[1]{^{(#1)}}
\def\d{{\texttt{d}}}

\newcommand{\diag}{\mathbb{D}}

\newcommand{\paragraphsp}[1]{\paragraph*{#1}}

\DeclareMathOperator{\Xone}{X_{\Circled{1}}}
\DeclareMathOperator{\Yone}{Y_{\Circled{1}}}
\DeclareMathOperator{\Zone}{Z_{\Circled{1}}}
\DeclareMathOperator{\Sone}{S_{\Circled{1}}}

\DeclareMathOperator{\Xtwo}{X_{\Circled{2}}}
\DeclareMathOperator{\Ytwo}{Y_{\Circled{2}}}
\DeclareMathOperator{\Ztwo}{Z_{\Circled{2}}}
\DeclareMathOperator{\Stwo}{S_{\Circled{2}}}

\begin{document}

\maketitle

\begin{abstract}

In 1960, Osborne proposed a simple iterative algorithm for matrix balancing with outstanding numerical performance. Today, it is the default preconditioning procedure before eigenvalue computation and other linear algebra subroutines in mainstream software packages such as Python, Julia, MATLAB, EISPACK, LAPACK, and more. 
Despite its widespread usage,
Osborne's algorithm has long resisted theoretical guarantees for its runtime: the first polynomial-time guarantees were obtained only in the past decade, and recent near-linear runtimes remain confined to variants of Osborne's algorithm with important differences that make them simpler to analyze but empirically slower. In this paper, we address this longstanding gap between theory and practice by proving that Osborne's original algorithm---the \emph{de facto} preconditioner in practice---in fact has a near-linear runtime. This runtime guarantee (1) is optimal in the input size up to at most a single logarithm, (2) is the first runtime for Osborne's algorithm that
does not dominate the runtime of downstream tasks like eigenvalue computation, and (3) improves upon the theoretical runtimes for all other variants of Osborne's algorithm.

\end{abstract}

\section{Introduction}

In his seminal 1960 paper, Osborne proposed an idea for matrix preconditioning that has since become ubiquitous~\citep{osborne1960pre}: prior to computing eigenvalues of a matrix $\mA$, one should first find a diagonal similarity transform $\mM = \mD\mA\mD^{-1}$ that is ``balanced'', i.e., such that the norm of its $i$-th row is equal to the norm of its $i$-th column, for all $i$. One should then run the downstream computations on $\mM$ rather than on $\mA$, the rationale being that diagonal scalings do not affect the eigenvalues, and standard algorithms for linear algebra are typically much more numerically stable for balanced matrices. 

\par Nowadays, this preconditioning step of ``matrix balancing'' is used by default before eigenvalue computation in mainstream software packages such as Python, Julia, MATLAB, LAPACK, and EISPACK~\citep{NumPyEig,JuliaEigen,MathWorksEig,anderson1999lapack,smith2013matrix}. 
Matrix balancing is also commonly used for preconditioning before other core linear algebra subroutines such as matrix exponentiation~\citep{ward1977numerical,higham2005scaling}. Further, it has since become a key component of applications beyond preconditioning, including in webpage ranking~\citep{tomlin2003new}, combinatorial optimization~\citep{altschuler2022approximating}, and economics and operations research~\citep{schneider1990comparative}. 

\par This widespread adoption is largely due to an efficient algorithm for matrix balancing---proposed by Osborne in the same 1960 paper~\citep{osborne1960pre}. In practice, this algorithm efficiently computes a diagonal scaling $\mD\mA\mD^{-1}$ that is approximately balanced, in time that is essentially negligible compared to downstream tasks like eigenvalue computation~\citep{press2007numerical}. Because of this, it has remained the default implementation in modern software packages.

\par Osborne's algorithm can be defined for any norm with which to balance row/column pairs~\citep{parlett1969balancing}. All of the aforementioned software packages use the $\ell_1$ or $\ell_2$ version\footnote{In fact, the standard advice is to use $\ell_1$ balancing since, up to a simple transformation, it is equivalent to $\ell_2$, while its simpler implementation makes it faster and more numerically stable~\citep{press2007numerical}. Remaining uses of $\ell_2$ balancing in existing implementations therefore appear to be historical artifacts.}, and actually these two versions of matrix balancing (and more generally all $\ell_p$, $p < \infty$) are equivalent as they are trivially reducible to each other (to see this, consider entrywise raising the matrix $\mA$ to the power $p$). See the related work for a lengthier discussion and for comments on the $\ell_{\infty}$ version of matrix balancing, which has an interesting line of work in its own right. In this paper, we restrict to the $\ell_1$ norm as this is the standard implementation in practice, and also because this simplifies notation: the $\ell_1$ norm simply amounts to row/column sums after replacing each entry of $\mA$ by its absolute value (without loss of generality, as this does not affect the optimal $\mD$). In this notation, the matrix balancing problem can be restated more simply as follows.

\begin{problem}\label{prob:prob}
	Given an $n \times n$ matrix $\mA$ with non-negative entries, find an $n \times n$ diagonal matrix $\mD$ with positive entries such that $\mM = \mD\mA\mD^{-1}$ is balanced in the sense that its row sums $r(\mM) := \mM\vone$ are equal to its column sums $c(\mM) := \mM^\top \vone$. Here, $\vone$ denotes the all-ones vector in $\R^n$. 
\end{problem}

Osborne's algorithm is simple to state\footnote{Implementations often also restrict elements of $\mD$ to powers of $2$ to avoid rounding errors, and pre-process $\mA$ by permuting its rows/columns into irreducible diagonal blocks which can then be separately balanced via Osborne's algorithm in parallel. For simplicity of exposition, we ignore both of these minor modifications since the former is just to prevent rounding errors, and the latter is easily performed in linear time~\citep{tarjan1972depth}.}. As is standard~\cite{press2007numerical}, 
here and henceforth suppose that the diagonal of $\mA$ is zeroed out, as this has no effect on the balancing $\mD$ and simplifies the statement of Osborne's algorithm. Osborne's algorithm initializes $\mD$ to the identity; this corresponds to no scaling. Then iteratively for each index $i$, it multiplies $\mD_{ii}$ by $\sqrt{c_i(\mM) / r_i(\mM)}$ so that the $i$\textsuperscript{th} row sum $r_i(\mM)$ equals the $i$\textsuperscript{th} column sum $c_i(\mM)$ of the current balancing $\mM = \mD\mA\mD^{-1}$. The update indices $i$ are chosen in cyclic order $i \in \{1, \dots, n\}$, until a convergence criterion is met, for instance when the normalized $\ell_1$-imbalance is below some threshold $\eps$:
\begin{align}
	\frac{\big\|r(\mM) - c(\mM)\big\|_1}{\sum_{ij}\mM_{ij}} \leq \eps\,.
	\label{eq:eps}
\end{align}
This notion of approximate balancing is useful for downstream applications, has several natural interpretations, and implies other notions studied in prior work; see \S\ref{sec:related-work} for a further discussion.

\par It has been a longstanding challenge to explain the practical efficacy of Osborne's algorithm. Early results in the 1960-1980s showed asymptotic convergence of the algorithm~\citep{osborne1960pre,grad1971matrix}, uniqueness of the balanced matrix~\citep{hartfiel1971concerning}, and characterized what matrices can be balanced~\citep{eaves1985line}. 
Only in the past decade did breakthrough works establish quantitative convergence rates for Osborne's algorithm, first for $\ell_{\infty}$ balancing~\citep{schulman2017analysis}, and then for $\ell_p$ balancing, $p < \infty$~\citep{ostrovsky2017matrix}. Ignoring the dependence on the approximation error $\epsilon > 0$ (which for most practical purposes is treated as a constant), these runtime bounds scale as roughly $mn^2$ in the dimension $n$ of the input matrix and its number of non-zero entries $m$. Note that an $mn^2$ runtime for matrix preconditioning dominates the roughly $\cO(n^3)$ runtime of eigenvalue computation, and thus would seem to defeat the purpose of being a cheap pre-processing step. Thus, although these breakthrough results establish polynomial runtimes for Osborne's algorithm, they do not explain its practical efficacy.

\par Towards this goal, recent work has established faster runtime bounds for \emph{variants} of Osborne's algorithm that choose the update indices $i$ in non-cyclic ways. In particular,~\citep{ostrovsky2017matrix} established $\tilde{\cO}(n^2)$ runtime bounds for variants in which the update index $i$ is chosen greedily (according to the maximum imbalance) or randomly (according to their norms), and~\citep{altschuler2023near} showed that uniform-random choices of indices provide $\tilde{\cO}(m)$ runtime bounds. These runtime bounds all scale inverse polynomially in the target error $\eps$ and logarithmically in the matrix conditioning $\kappa := \sum_{ij} \mA_{ij} / \min_{ij : \mA_{ij} \neq 0} \mA_{ij}$;
see the related work \S\ref{sec:related-work} for details. However, despite these improved runtimes for variants of Osborne's algorithm, the original proposal of cyclic updates remains the de facto implementation. This is for many reasons: in practice, Osborne's original algorithm typically converges as fast (if not faster, see \S\ref{sec:num-ex} for numerical examples), it is cache-friendly (since it accesses adjacent rows/columns), it requires no overhead (the greedy variant requires additional data structures to amortize the cost of greedily choosing update indices), it exploits sparsity of the input matrix (leading to empirical runtimes scaling in $m$ rather than $n^2$),
and, as we show, it can be safely implemented with polylogarithmic bit complexity (which is exponentially better than for the greedy method).

\par Another line of work has sought to develop entirely different algorithms for matrix balancing, by solving a convex optimization reformulation of the matrix balancing problem; see the related work in \S\ref{sec:related-work}. However, these algorithms are not practical (at least currently) due to their reliance on complicated subroutines such as the Ellipsoid algorithm~\citep{kalantari1997complexity} and Laplacian solvers~\citep{cohen2017matrix,allen2017much}.

\par The literature on matrix balancing, therefore, suffers from a longstanding gap between theory and practice. In practice, Osborne's original algorithm outperforms all other competitors, has remained the  implementation of choice for over 60 years, and appears to run in roughly near-linear time $\tilde{\cO}(m)$~\citep[\S11.6.1]{press2007numerical}.  Such a runtime \emph{would} finally explain the practical efficacy of Osborne's algorithm as, in particular, it would be negligible compared to the runtime of downstream tasks like eigenvalue computation. %

\subsection{Contribution}\label{ssec:intro:overview}

We address this longstanding open problem by proving a near-linear runtime bound for Osborne's original algorithm. This provides a roughly $\cO(n^2)$ improvement over the previously best-known runtime for Osborne's original algorithm~\citep{ostrovsky2017matrix}. We emphasize that this runtime not only is optimal in the input size $m$ (up to at most a single logarithm),
but moreover is the first runtime that does not dominate the $\cO(n^3)$ runtime of downstream tasks like eigenvalue computation---as one would hope for a preconditioner. Our main result is summarized in the following theorem, where $\d$ denotes the diameter of the graph $G_\mA = ([n], \{(i,j) : \mA_{ij} \neq 0\})$ and we recall 
$\kappa = \sum_{ij} \mA_{ij} / \min_{ij : \mA_{ij} \neq 0} \mA_{ij}$ controls the conditioning of $\mA$. 

\begin{restatable}[Main result]{theorem}{mainThm}\label{thm:main}
Given a balanceable matrix $\mA$ and accuracy $\epsilon > 0$, Osborne's algorithm finds an $\eps$-balancing after $\cO\big(\frac{\log \kappa}{\epsilon}\min\{\frac{1}{\epsilon}, \d\}\big)$ cycles. This amounts to $\cO\big(m \frac{\log \kappa}{\epsilon}\min\{\frac{1}{\epsilon}, \d\}\big)$ arithmetic operations.
\end{restatable}

Our result matches or improves the best-known runtime for any variant of Osborne's algorithm, in all problem parameters. In particular, it matches the previous state-of-the-art bound~\citep{altschuler2023near}, which applies only to the less practically effective variant of Osborne's algorithm that uses random (with replacement) choices of update indices $i$, and, further, our result provides a stronger guarantee since it is deterministic while theirs only holds in expectation or with high probability.

\par Note that our runtime is \emph{near}-linear in the input sparsity $m$ only because of a single $\log m$ factor implicit in $\log \kappa$. The result applies to arbitrary input matrices $\mA$ that are balanceable, which by the classical characterization of~\citep{eaves1985line} is equivalent to the graph $G_\mA$ being strongly connected, i.e., having diameter $\d < \infty$. The runtime dependence on the precision $\eps$ automatically accelerates from $\eps^{-2}$ to $\eps^{-1}$ when $\d$ is small, which is a natural quantitative measure of how balanceable the input $\mA$ is. Note that $\d$ is small under mild conditions, e.g., $\d = 1$ if $\mA$ has at least one strictly positive row/column pair, and $\d = \tilde{\cO}(1)$ with high probability if $\mA$ has a random sparsity pattern where each entry is non-zero independently with probability $\Omega((\log n)/n)$~\citep{gilbert1959random}.

We remark that our analysis has several additional benefits. 1) Theorem~\ref{thm:main} holds even when exact arithmetic is replaced by arithmetic on numbers with bit complexity $\cO(\log (n \kappa / \eps))$. This is an exponential improvement over the $\cO(n \log \kappa)$ bit complexity required by the previous state-of-the-art for Osborne's original algorithm~\citep{ostrovsky2017matrix}, and constitutes an additional $\tilde{\cO}(n)$ improvement in runtime beyond the aforementioned $\cO(mn)$ improvement in the number of arithmetic operations. This theoretical finding that Osborne's algorithm is stable with respect to low bit complexity implementations also justifies practical implementations which, since 1969, have rounded entries of $\mD$ to powers of the radix base~\citep{parlett1969balancing}. 2) Our analysis holds for arbitrary orderings of $\{1, \dots, n\}$ in each cycle, even if the orderings change between cycles. This leads to improved runtimes for several variants of Osborne's algorithm as immediate corollaries. As one example, our bound extends to the random-reshuffle variant of cyclic Osborne studied in~\citep{altschuler2023near}, resulting in an $\cO(n)$ improvement for that variant. As another example, our bound extends to parallelized variants of Osborne's algorithm, requiring the same amount of total work as in Theorem~\ref{thm:main}, but improving the runtime (in terms of rounds) from near-linear in the input size $m$ to near-linear in the chromatic number of $G_\mA$ (when given a coloring). See \S\ref{ssec:analysis:discussion} for details on these extensions.

\paragraphsp{Analysis overview.} 
Our starting point is a well-known formulation of matrix balancing as a convex optimization problem and a corresponding interpretation of Osborne's algorithm as an exact coordinate descent algorithm for this optimization problem (see the preliminaries section \S\ref{sec:prelim}). A key benefit of this perspective is that the optimization objective provides a principled way to track the progress of Osborne's algorithm---a non-obvious question since updating the $i$-th index fixes the imbalance in the $i$-th row/column pair, but potentially imbalances all other row/column pairs. 
This convex optimization perspective was at the core of the recent analyses of greedy and randomized variants of Osborne's algorithm \citep{ostrovsky2017matrix,altschuler2023near} and alternative algorithms \citep{kalantari1997complexity,cohen2017matrix,allen2017much}, 
since it allows arguing that each iteration makes progress decreasing the optimization objective, after which one could simply sum the per-iteration progress. 

\par The core challenge for analyzing Osborne's original algorithm---which \emph{cyclically} chooses update indices $i$ rather than greedily or randomly---is that the per-iteration progress is affected by where the selected index appears in the cycle. This is the same conceptual challenge that arises in proving complexity bounds more generally for  cyclic coordinate descent in convex optimization \cite{beck2013convergence,wright2020analyzing,song2021coder,lin2023accelerated,gurbuzbalaban2017cyclic,wright2015coordinate}. Note that none of these past results concerning cyclic algorithms leads to a near-linear runtime for Osborne's algorithm, since their Lipschitz assumptions and worst-case superlinear complexity bounds are incompatible with the convex optimization formulation of matrix balancing and the desired convergence guarantees. (See also the discussion and open problem statement in, e.g.,~\citep{altschuler2023near}.)

Due to the dependencies between iterations in a cycle, previous analyses for cyclic Osborne's algorithm were unable to argue about the progress of a cycle beyond the first (few) iterations in it~\citep{ostrovsky2017matrix,altschuler2023near}, leading to pessimistic runtime bounds that did not account for most of the algorithm's progress. 
In contrast, our analysis obtains near-linear runtimes by directly analyzing the cumulative progress over an entire cycle. %
On a conceptual level, this tracking of aggregate progress over a full cycle is inspired by recent advances in cyclic coordinate optimization methods \cite{song2021coder,chakrabarti2023block,cai2022cyclic,lin2023accelerated}, although on the technical level our result has no obvious connection to this line of work. %

\par Along these lines, a key lemma we show is that after completing a cycle, the imbalance of the currently scaled matrix is bounded by the sum, over indices $i$, of the imbalance $|r_i(\cdot) - c_i(\cdot)|$ of the $i$-th row/column pair at the time it was corrected in that cycle (formal statement in Lemma~\ref{lem:imbalance}). Our analysis then proceeds by considering two cases. If all coordinates $i$ had ``small'' imbalance when updated, then by this lemma, the current matrix has small imbalance, thus Osborne's algorithm would successfully terminate. Otherwise, some coordinates must have had ``large'' imbalance when they were updated, thus those iterations of Osborne's algorithm must have dramatically improved the optimization objective. Since the objective gap between initialization and optimality is bounded, such improvements could not happen too many times, hence Osborne's algorithm could not have had many such iterations before termination. Full details in \S\ref{sec:analysis}. 

\subsection{Further Related Work}\label{sec:related-work}

\paragraphsp{State-of-the-art runtimes for variants of Osborne's algorithm.}
As mentioned above, polynomial runtimes for Osborne's algorithm and its variants were not known until the past decade. Prior to this work, the state-of-the-art bound for Osborne's original algorithm, which cycles through update indices $i \in \{1, \dots, n\}$ in sequential order, was $\cO(mn^2 \eps^{-2} \log \kappa)$ arithmetic operations~\citep{ostrovsky2017matrix}.~\citep{altschuler2023near} showed an improved complexity of $\cO(mn \epsilon^{-1} \log(\kappa)\min\{\epsilon^{-1}, \d\})$ for a \emph{random-reshuffle} variant of cyclic Osborne, which uses randomly permuted orderings in each cycle. A key obstacle for prior work is the analysis of cyclic updating schemes; faster runtimes have been proven for other variants of Osborne's algorithm.~\citep{ostrovsky2017matrix} proved an $\tilde{\cO}(n^2 \eps^{-2} \log \kappa)$ bound on the number of arithmetic operations for variants of Osborne's algorithm that choose update indices $i$ either greedily (by maximizing the current imbalance $(\sqrt{r_i} - \sqrt{c_i})^2$) or randomly (by sampling $i$ with probability proportional to $r_i + c_i$).~\citep{altschuler2023near} showed that choosing indices $i$ uniformly at random achieves a near-linear runtime, requiring only $\cO(m \eps^{-1} \log(\kappa) \min\{ \eps^{-1}, \d\})$ arithmetic operations (in expectation or with high probability). It is worth mentioning that the greedy variant and the state-of-the-art result for Osborne's original algorithm in~\citep{ostrovsky2017matrix} both require operating on numbers with bit complexity $\cO(n \log \kappa)$.~\citep{ostrovsky2017strictly} considered more sophisticated variants of Osborne's algorithm to achieve stricter notions of balancing and showed an $\tilde{\cO}(n^{19} \eps^{-4} \log^4 \kappa)$ runtime bound. %

\paragraphsp{Alternative algorithms.} There is an orthogonal line of research that aims to develop alternative algorithms for matrix balancing by directly solving a convex optimization reformulation. In this direction,~\citep{kalantari1997complexity} proved the first polynomial runtime of $\tilde \cO(n^4\log((\log\kappa)/\epsilon))$ using the Ellipsoid algorithm, and recently~\citep{cohen2017matrix} used Laplacian solvers to achieve a runtime of $\tilde \cO(m^{3/2}\log(\kappa/\epsilon))$ using an interior-point algorithm and a runtime of $\tilde \cO(m \d \log^2(\kappa / \epsilon) \log \kappa)$ using a Newton-type algorithm. See also~\citep{allen2017much} for similar results. However, these algorithms are not practical (at least currently) due to their use of complicated subroutines. Moreover, our result for Osborne's algorithm is also theoretically stronger for practically relevant regimes because (i) for the purpose of matrix preconditioning, matrix balancing is performed to extremely low accuracy $\eps = \cO(1)$~\citep{press2007numerical,parlett1969balancing}, and (ii) our dependence on other parameters is better, namely $\cO(m \log \kappa)$ rather than $\tilde{\cO}(m \d \log^3 \kappa)$ (note the extra factor of $\d$, cubic dependence on $\log \kappa$, and the multiple suppressed logarithms due to the heavy use of Laplacian solvers).

\paragraphsp{Max-balancing.} As mentioned above, the matrix balancing problem can be posed for any norm with which to compare rows and columns. The $\ell_{\infty}$ version of this problem has its own interesting line of work and was also considered in Osborne's original paper~\citep{osborne1960pre}. This problem is sometimes called max-balancing, as it requires balancing the maximum entry in each row/column pair. For this setting, asymptotic convergence for Osborne's algorithm was not proven until~\citep{chen2000balancing}, and only recently did the breakthrough paper~\citep{schulman2017analysis} provide the first polynomial runtime and characterize what matrices have a unique balancing. Their runtime scales as $\cO(mn^2)$ in the input size, which is tight for sparse matrices due to a known $\Omega(n^3)$ lower bound for Osborne's algorithm, at least for high-precision guarantees where $\eps = n^{-\Omega(1)}$~\citep{chen1998balancing}.
Alternative combinatorial algorithms for max-balancing have also been proposed:~\citep{schneider1991max} gave an $\cO(n^4)$ algorithm, and~\citep{young1991faster} gave an $\tilde{\cO}(mn + n^2)$ algorithm. It is worth emphasizing that although $\ell_{\infty}$ balancing might seem syntactically similar to $\ell_p$ balancing for $p < \infty$, the non-smoothness of the max function fundamentally changes many aspects of the problem, including the nature of solutions (e.g., the max-balancing solution is not unique~\citep{chen2000balancing}) and the nature of algorithms (e.g., it appears that analyses of Osborne's algorithm do not generalize across these settings~\citep{schulman2017analysis,ostrovsky2017matrix,altschuler2023near}).

\paragraphsp{Coordinate descent.} While, as noted before, Osborne's algorithm can be interpreted as cyclic coordinate descent performing exact minimization over each selected coordinate, none of the existing results from the optimization literature on cyclic coordinate methods~\citep{beck2013convergence,lee2020inexact,wright2020analyzing,song2021coder,cai2022cyclic,lin2023accelerated,chakrabarti2023block,gurbuzbalaban2017cyclic,wright2015coordinate} lead to the near-linear runtime obtained in this paper. This is for several reasons. One is that most of those analyses rely upon ``gradient-descent-style'' updates rather than exact minimization updates that Osborne's algorithm makes. Another is that those convergence results rely on coordinate smoothness conditions that would, at best, result in polynomial (but crucially superlinear) runtime. For more details and other related work on randomized coordinate methods, we refer to the discussion in~\citep{altschuler2023near}. %

\paragraphsp{Termination criteria.} 
    Throughout, we use the $\ell_1$ termination criterion~\eqref{eq:eps}. This criterion is the focus of previous state-of-the-art results for variants of Osborne's algorithm~\citep{altschuler2023near} as it implies the weaker $\ell_2$ criterion studied in other works (e.g.,~\citep{ostrovsky2017matrix,allen2017much,cohen2017matrix,kalantari1997complexity}, which is motivated by the fact that the numerical stability of eigenvalue calculations depends on the Frobenius norm of the balancing, see~\citep{kalantari1997complexity,osborne1960pre}) and also appears to be more useful in applications due to natural graph theoretic interpretations (such as netflow imbalance) and probabilistic interpretations (such as total variation imbalance), see, e.g.,~\citep[Remark 1]{altschuler2023near} and~\citep[Remarks 2.1 and 5.8]{altschuler2022approximating}. 
    In practice, a standard termination criterion for Osborne's algorithm is $2\sqrt{r_j c_j} > 0.95 (r_j + c_j)$ \citep{parlett1969balancing}, where $r_j, c_j$ are the row/column sums computed within a cycle, right before coordinate $j$ is updated. Observe that this criterion does not actually imply any notion of balancing, since these row/column sums are all computed at different iterations (i.e., with different diagonal scalings). Instead, it appears this condition is used purely for engineering reasons: the within-cycle sums $r_j, c_j$ are already computed when performing the update in Osborne's algorithm (see Algorithm \ref{alg:ccd}), while computing the actual imbalance at the end of the cycle (as in \eqref{eq:eps}) requires an additional full pass over the input, which increases the runtime by a (small) constant factor. Interestingly, our analysis helps justify the use of this common practical criterion, by showing that it implies constant-factor $\ell_1$ balancing as in \eqref{eq:eps} (see the proof of case \Circled{1} in Proposition \ref{prop:grad-bound}). 

\paragraphsp{Sinkhorn's algorithm for matrix scaling.} A related but different problem is matrix scaling: given a non-negative rectangular matrix $\mA \in \R^{m \times n}$, target row sums $\vr$, and target column sums $\vc$, find diagonal matrices $\mD_1$ and $\mD_2$ such that $\mM = \mD_1 \mA \mD_2$ satisfies $r(\mM) = \vr$ and $c(\mM) = \vc$. Just like matrix balancing, matrix scaling has a rich history due to its many applications across multiple communities and has been popularized in large part due to the availability of a simple iterative algorithm that is effective in practice; see the surveys~\citep{idel2016review,peyre2019computational}. That algorithm, often called Sinkhorn's algorithm~\citep{sinkhorn1967diagonal}, has received renewed recent interest due to applications to optimal transport~\citep{cuturi2013sinkhorn} and has also been shown to converge in near-linear time~\citep{altschuler2017near}. Sinkhorn's algorithm solves matrix scaling by iteratively correcting all row sums or all columns sums, which corresponds to alternately performing an exact coordinate descent update over all variables in $\mD_1$, or over all variables in $\mD_2$. Notice that for Sinkhorn's algorithm, updates to the different entries in $\mD_1$ do not affect each other since they affect different row sums, and similarly for $\mD_2$; in contrast, for Osborne's algorithm, every update affects every other update. This is the key reason why Osborne's algorithm requires more sophisticated arguments about \emph{cyclic} coordinate descent, and is why proving a near-linear runtime is substantively more difficult than for Sinkhorn's algorithm. See the survey~\citep{altschuler2022flows} for a further discussion of the connections and differences.

\section{Preliminaries}\label{sec:prelim}

\paragraph{Notation.}

Throughout, $\mA \in \R^{n \times n}$ denotes the matrix to balance. As discussed in the introduction, without loss of generality $\mA$ is assumed to have non-negative entries and zero diagonal. We also assume without loss of generality that $\mA$ is balanceable, since it is well-known that approximately balancing a non-balanceable matrix can be reduced in linear time to approximately balancing balanceable matrices~\citep{chen2000balancing,cohen2017matrix}. 

Throughout, we reserve the notation $m$ for the number of nonzeros in $\mA$, $\kappa = \sum_{ij}\mA_{ij} / \min_{ij: \mA_{ij} \neq 0} \mA_{ij}$ for the conditioning of $\mA$, and $\d$ for the diameter of the graph $G_\mA = ([n], \{(i,j) : \mA_{ij} \neq 0\})$ associated to $\mA$. 
We denote the row and column sums of a matrix $\mA$ by $r(\mA) = \mA\bm{1}$ and $c(\mA) = \mA^\top\bm{1}$, respectively, and write $r_i(\mA) = \sum_{j = 1}^n \mA_{ij}$ and $c_i(\mA) = \sum_{j = 1}^n \mA_{ji}$ for the $i$\textsuperscript{th} row and column sum.
We denote the set $\{1, 2, \dots, n\}$ by $[n]$, and we denote the $i^{\rm th}$ coordinate of a vector $\vu$ by $\vu\bl{i}$.

\paragraphsp{Convex optimization reformulation of matrix balancing.}%

The starting point of our analysis is a well-known formulation of matrix balancing in terms of the following convex optimization problem:
\begin{equation}\label{eq:function}
    \min_{\vu \in \R^n}\Big\{\varphi(\vu) := \sum_{i=1}^n \sum_{j=1}^n e^{\vu\bl{i} - \vu\bl{j}} \mA_{ij}\Big\}.
\end{equation}
To understand this connection, let $\diag(e^{\vu})$ denote the diagonal matrix with entries $ e^{\vu\bl{1}}, e^{\vu\bl{2}}, \dots, e^{\vu\bl{n}}$, and observe that
\begin{align}\label{eq:gradient}
    \nabla \varphi(\vu) = r\big(\diag(e^\vu) \mA \diag(e^{-\vu})\big) - c\big(\diag(e^\vu) \mA \diag(e^{-\vu})\big).
\end{align}
Thus, $\nabla \varphi(\vu) = \mathbf{0}$ if and only if $\diag(e^{\vu})$ balances $\mA$. A similar argument extends to approximate matrix balancing: the $\epsilon$-imbalance as defined in~\eqref{eq:eps} corresponds to the $\epsilon$-stationarity of $\log \varphi(\vu)$ in the $\ell_1$-norm.
These well-known connections are summarized in the following lemma (see, e.g,. \cite{kalantari1997complexity} for proofs).
\begin{lemma}\label{lem:prelim}
For any balanceable matrix $\mA$, the following statements hold.
\begin{enumerate}
    \item The function $\varphi$ is log-convex over $\R^n$.
    \item The diagonal matrix $\diag(e^{\vu_*})$ balances $\mA$ if and only if $\vu_* \in \argmin \varphi$. 
    \item For any $\eps > 0$, the diagonal matrix $\diag(e^{\vu})$ $\epsilon$-balances $\mA$ if and only if $\|\nabla \varphi(\vu)\|_1 / \varphi(\vu) \leq \epsilon$.
\end{enumerate}
\end{lemma}

Our analysis also makes use of the following known helper lemmas, which respectively bound the initial potential and the $\ell_\infty$-distance of nontrivial balancings to minimizers of $\varphi$. See, e.g., \cite[Lemma 3 and Corollary 7]{altschuler2023near} for proofs. 
\begin{lemma}\label{lemma:bound-gap}
    For any balanceable matrix $\mA$, if $\varphi_* = \min_{\vu} \varphi(\vu)$ then $\varphi(\mathbf{0}) / \varphi_* \leq \kappa$.
\end{lemma}

\begin{lemma}\label{lemma:bound-dist}
    If $\vu \in \R^n$ satisfies $\varphi(\vu) \leq \varphi(\mathbf{0})$, then there exists a minimizer $\vu_* \in \argmin \varphi$ such that $\|\vu - \vu_*\|_\infty \leq \d \log \kappa$.
\end{lemma}

\paragraphsp{Interpretation of Osborne's algorithm as exact coordinate descent.} 
The reformulation of matrix balancing as a convex optimization problem $\min_{\vu} \varphi(\vu)$ enables interpreting Osborne's algorithm as a convex optimization algorithm for minimizing $\varphi$, namely the exact coordinate descent algorithm with cyclic updates. 
 The idea behind this connection is as follows. First, recall that Osborne's algorithm initializes the diagonal scaling to the identity---this corresponds to initializing the optimization algorithm at $\vu_0 = \mathbf{0}$. Then, more substantially, Osborne's algorithm cycles through indices $j \in \{1, \dots, n\}$, updating the $j$\textsuperscript{th} diagonal element of $\mD$ to balance the $j$\textsuperscript{th} row/column sums of $\mD \mA \mD^{-1}$. By~\eqref{eq:gradient}, this corresponds to minimizing $\varphi$ by updating the $j$\textsuperscript{th} coordinate of the current iterate $\vu$ (since $\mD = \diag(e^{\vu})$) while keeping all other coordinates fixed---this is by definition the exact coordinate descent algorithm for minimizing $\varphi$. 

\par This connection is summarized in Algorithm~\ref{alg:ccd}. That pseudocode also defines helpful notation for referring to the iterates of Osborne's algorithm: we let $k$ denote the iteration counter for the cycle (outer loop in Algorithm~\ref{alg:ccd}), $j$ denote the update index (inner loop in Algorithm~\ref{alg:ccd}), and $\vu_{k, j}$ denote the iterate before balancing the $j$\textsuperscript{th} row/column. For shorthand, we also write $\vu_{k} = \vu_{k,1}$ to denote the iterate obtained after completing $k$ cycles. Note that in this notation, the update in Osborne's algorithm for balancing the $j$\textsuperscript{th} row/column pair in cycle $k$ amounts to the exact coordinate descent step 
\begin{align*}
    \vu_{k,j+1} = \argmin_{\vu \in \{\vu_{k,j} + h \ve_j \; : \; h \in \R \}}\varphi(\vu)\,,
\end{align*}
which amounts to updating the $j$\textsuperscript{th} entry via $\vu_{k + 1}\bl{j} = \vu_k\bl{j} + \frac{1}{2}\log c_j(\vu_{k, j}) - \frac{1}{2}\log r_j(\vu_{k, j})$, or equivalently updating the $j$\textsuperscript{th} diagonal entry of $\diag(e^{\vu})$ so that
\begin{equation}\label{eq:osborne-step}
    e^{\vu_{k+1}\bl{j} - \vu_k\bl{j}} = \sqrt{c_j(\vu_{k,j}) / r_j(\vu_{k, j})} \,.
\end{equation}
This is precisely the update step for Osborne's algorithm. For ease of recall, we record the simple fact that this update indeed balances the $j$\textsuperscript{th} row/column sums by making them equal to their geometric mean.

\begin{fact}\label{lem:step}
    Let $\{\vu_{k, j}\}$ 
    denote the iterates of Osborne's algorithm (Algorithm \ref{alg:ccd}). Then for every $k \geq 0$ and $j \in [n],$
     \begin{align*}
         r_j(\vu_{k, j + 1}) = c_j(\vu_{k, j + 1}) = \sqrt{r_j(\vu_{k, j})c_j(\vu_{k, j})}.
     \end{align*}
\end{fact}

\begin{algorithm}[tb]
\caption{Osborne's algorithm as Cyclic Exact Coordinate Descent}
\label{alg:ccd}
\begin{algorithmic}
   \STATE {\bfseries Input:} matrix $\mA \in \R^{n \times n}_{\geq 0}$, number of cycles $K$
   \STATE{$\vu_0 = \mathbf{0}$}
   \FOR{$k=0$ {\bfseries to} $K - 1$}
   \STATE $\vu_{k, 1} = \vu_{k}$
   \FOR{$j=1$ {\bfseries to} $n$}
   \STATE $\vu_{k + 1}\bl{j} = \vu_k\bl{j} + \frac{1}{2}\log c_j(\vu_{k, j}) - \frac{1}{2}\log r_j(\vu_{k, j})$
   \STATE $\vu_{k, j + 1} = \argmin_{\vu \in \{\vu_{k,j} + h \ve_j \; : \; h \in \R \}}\varphi(\vu) =\big(\vu_{k+1}\bl{1}, \dots, \vu_{k+1}\bl{j}, \vu_k\bl{j + 1}, \dots, \vu_k\bl{n}\big)^\top$
   \ENDFOR
   \ENDFOR
   \STATE Return $\vu_K$
\end{algorithmic}
\end{algorithm}

\section{Near-Linear Runtime of Osborne's Algorithm}\label{sec:analysis}

In this section we establish the claimed near-linear runtime for Osborne's original algorithm, proving Theorem~\ref{thm:main}. See \S\ref{ssec:intro:overview} for a high-level overview of the analysis strategy, in particular how it exploits the convex optimization perspective that Osborne's algorithm is performing exact coordinate descent to minimize $\varphi$, and how we deal with the cyclic nature of these updates.

\subsection{Proof of Theorem~\ref{thm:main}}

Our starting point is a ``descent lemma'' which relates the progress $\varphi(\vu_{k+1}) - \varphi(\vu_k)$ of Osborne's algorithm from a cycle of updates, to each row/column imbalance 
right before it was updated within the cycle. 
This descent lemma is known~\cite{ostrovsky2017matrix}; we provide a short proof for the convenience of the reader. 
        
\begin{lemma}[Descent Lemma]\label{lem:descent}
    Let $\{\vu_{k, j}\}$ denote the iterates of Osborne's algorithm (Algorithm~\ref{alg:ccd}). 
    Then, for all $k \geq 0$ and $j \in [n]$,
    \begin{equation}\notag
    \begin{aligned}
        \varphi(\vu_{k, j + 1}) - \varphi(\vu_{k, j}) 
        = - \Big(\sqrt{r_j(\vu_{k, j})} - \sqrt{c_j(\vu_{k, j})}\Big)^2. 
    \end{aligned}
    \end{equation}
    Hence, for all $k \geq 0,$ the descent of the potential over the full cycle is
    \begin{equation}\notag
    \begin{aligned}
        \varphi(\vu_{k+1}) - \varphi(\vu_k) 
        = - \sum_{j = 1}^n \Big(\sqrt{r_j(\vu_{k, j})} - \sqrt{c_j(\vu_{k, j})}\Big)^2. 
    \end{aligned}
    \end{equation}
\end{lemma}
\begin{proof}
    For the first claim, observe that
    \begin{align*}
         \varphi(\vu_{k, j+1}) - \varphi(\vu_{k, j}) 
         &= r_j(\vu_{k, j + 1}) + c_j(\vu_{k, j + 1}) - r_j(\vu_{k, j}) - c_j(\vu_{k, j})
         \\ &= 2\sqrt{r_j(\vu_{k, j})c_j(\vu_{k, j})} - r_j(\vu_{k, j}) - c_j(\vu_{k, j}) 
         \\ &=- \Big(\sqrt{r_j(\vu_{k, j})} - \sqrt{c_j(\vu_{k, j})}\Big)^2.
    \end{align*}
    Above, the first step is because  $\vu_{k, j+1}$ and $\vu_{k, j}$ only differ on the $j$\textsuperscript{th} coordinate, which only affects the $j$\textsuperscript{th} row/column sums in $\varphi(\vu) = \sum_{ij} (\diag(e^{\vu}) \mA \diag(e^{-\vu}))_{ij}$. The second step is by Fact~\ref{lem:step}, which ensures that $r_j(\vu_{k, j + 1}) = c_j(\vu_{k, j + 1}) = \sqrt{r_j(\vu_{k, j})c_j(\vu_{k, j})}$ after balancing the $j$\textsuperscript{th} row/column sum. 
    \par For the second claim, re-write the cycle's progress as the telescoping sum
    \begin{equation*}
        \varphi(\vu_{k+1}) - \varphi(\vu_k) = \sum_{j=1}^n \big(\varphi(\vu_{k, j+1}) - \varphi(\vu_{k, j})\big)\,,
    \end{equation*}
    and then apply the first claim to each summand.
\end{proof}

The rest of our analysis is new. Although the above descent lemma was previously known, a key difficulty for using it is that the ``descent'' involves row/column imbalances at different times during the cycle, each corresponding to a different solution vector $\vu_{k, j}$. Previous work was unable to control the behavior at different iterations within the cycle since they affect each other, leading to analyses that only made use of progress from the first (few) iterations, thus missing out on the bulk of the progress of Osborne's algorithm in each cycle, and ultimately leading to pessimistic runtime bounds.

\par The following key lemma enables us to argue about all iterations within the cycle. Specifically, this lemma bounds the row/column imbalance $\|\nabla \varphi(\vu_{k + 1})\|_1$ at the end of a cycle, by the imbalance of each row/column pair right before it was updated in the cycle. The former is relevant for the termination criteria for Osborne's algorithm, and the latter can be related to the within-cycle row/column imbalance in the descent lemma, as described below.

For the proof, it is helpful to define the following notational shorthand for partial row sums and partial column sums: for $j \in [n]$ and $1 \leq s \leq t \leq n$, let
\begin{equation}\label{eq:partial-row-column-sums}
\begin{aligned}
    & r_j^{s:t}(\vu) := 
    \sum_{i=s}^t \big(\diag(e^\vu) \mA \diag(e^{-\vu})\big)_{ji}
    =
    \sum_{i=s}^t e^{\vu\bl{j} - \vu\bl{i}}\mA_{ji}\,, 
    \\
    & c_j^{s:t}(\vu) :=  \sum_{i=s}^t \big(\diag(e^\vu) \mA \diag(e^{-\vu})\big)_{ij} = \sum_{i=s}^t e^{\vu\bl{i} - \vu\bl{j}}\mA_{ij}\,. 
\end{aligned}
\end{equation}
Note in particular that since $\mA_{ii} = 0$,
\begin{equation}\label{eq:partial-row-column-sums-2}
\begin{aligned}
    r_j(\vu) = r_j^{1:j-1}(\vu) + r_j^{j+1:n}(\vu)\,, & \\
    c_j(\vu) = c_j^{1:j-1}(\vu) + c_j^{j+1:n}(\vu)\,.
\end{aligned}
\end{equation}

\begin{lemma}[Imbalance Lemma]\label{lem:imbalance}
     Let $\{\vu_{k, j}\}$ denote the iterates of Osborne's algorithm (Algorithm~\ref{alg:ccd}). Then, for every cycle $k \geq 0$,
    \begin{equation}\label{eq:imbalance-lem}
        \|\nabla \varphi(\vu_{k+1})\|_1 
        \leq \sum_{j=2}^n \big|r_j(\vu_{k, j}) - c_j(\vu_{k, j})\big|. 
    \end{equation}
\end{lemma}
\begin{proof}
    Fix any $j \in [n]$. Since $\vu_{k,j+1}$ is constructed by balancing the $j$\textsuperscript{th} row and column, Fact~\ref{lem:step} ensures that
    \begin{align*}
        r_j(\vu_{k,j+1}) = c_j(\vu_{k,j+1})\,.
    \end{align*}
    By~\eqref{eq:partial-row-column-sums-2} and the fact that $\vu_{k,j+1}^{(i)} = \vu_{k+1}^{(i)}$ for all coordinates $i \leq j$ (see the last line of Algorithm~\ref{alg:ccd}),
    \begin{align*}
        r_j(\vu_{k,j+1}) 
        &=
        r_j^{1:j-1}(\vu_{k,j+1}) 
        + 
        r_j^{j+1:n}(\vu_{k,j+1}) 
        =
        r_j^{1:j-1}(\vu_{k+1}) 
        + 
        r_j^{j+1:n}(\vu_{k,j+1})\,,
    \end{align*}
    By an analogous argument for the column sums,
    \begin{align*}
    c_j(\vu_{k,j+1}) 
        &=
        c_j^{1:j-1}(\vu_{k,j+1}) 
        + 
        c_j^{j+1:n}(\vu_{k,j+1}) 
        =
        c_j^{1:j-1}(\vu_{k+1}) 
        + 
        c_j^{j+1:n}(\vu_{k,j+1})\,.
    \end{align*}
    Combining the above three displays and re-arranging gives
    \begin{align}\label{eq:imbalance-1}
        r_j^{1:j-1}(\vu_{k+1}) 
        -
        c_j^{1:j-1}(\vu_{k+1})
        = 
        c_j^{j+1:n}(\vu_{k,j+1}) - r_j^{j+1:n}(\vu_{k,j+1})\,.
    \end{align}

    \par Now, by applying in order~\eqref{eq:gradient}, then~\eqref{eq:partial-row-column-sums-2}, and then~\eqref{eq:imbalance-1}, the $j$\textsuperscript{th} coordinate of the gradient $\nabla \varphi$ after the completion of cycle $k$ is equal to
    \begin{align}
       &\;  \nabla_j \varphi(\vu_{k+1}) \nonumber \\
        =\; &
        r_j(\vu_{k+1}) - c_j(\vu_{k+1}) \nonumber
        \\ =\;& r_j^{1: j - 1}(\vu_{k+1}) + r_j^{j + 1: n}(\vu_{k+1})
        -
        c_j^{1: j - 1}(\vu_{k+1}) - c_j^{j + 1: n}(\vu_{k+1}) \nonumber
        \\ =\; & \big(r_j^{j+1:n}(\vu_{k+1}) - r_j^{j+1:n}(\vu_{k, j + 1})\big) - \big(c_j^{j+1:n}(\vu_{k+1}) - c_j^{j+1:n}(\vu_{k, j + 1})\big). \label{eq:imbalance-decompose}
    \end{align}
    Thus, by using in order the definition of the $\ell_1$ norm, the above display and the definition of the partial row/column sums, the triangle inequality, swapping sums, and then again the definition of the partial row/column sums, we have 
     \begin{align*}
        \;& \|\nabla \varphi(\vu_{k+1})\|_1 
        \\  = \;& \sum_{j=1}^n |\nabla_j \varphi(\vu_{k+1})|
        \\  = \;& \sum_{j=1}^n \Bigg|\sum_{i=j+1}^n \Big(\big(e^{\vu_{k}\bl{i} - \vu_{k + 1}\bl{i}} - 1\big)e^{\vu_{k+1}\bl{j} - \vu_k\bl{i} }\mA_{ji} - \big(e^{\vu_{k + 1}\bl{i} - \vu_{k}\bl{i}} - 1\big)e^{\vu_{k}\bl{i} - \vu_{k+1}\bl{j}}\mA_{ij}\Big)\Bigg|
        \\
        \leq \;& \sum_{j=1}^{n} \sum_{i=j+1}^n \Big(\big|e^{\vu_{k}\bl{i} - \vu_{k + 1}\bl{i}} - 1\big|e^{\vu_{k+1}\bl{j} - \vu_k\bl{i} }\mA_{ji} + \big|e^{\vu_{k + 1}\bl{i} - \vu_{k}\bl{i}} - 1\big|e^{\vu_{k}\bl{i} - \vu_{k+1}\bl{j}}\mA_{ij}\Big) \\
       = \;& \sum_{i=2}^n \Big( \big|e^{\vu_{k}\bl{i} - \vu_{k + 1}\bl{i}} - 1\big| \sum_{j = 1}^{i-1} e^{\vu_{k+1}\bl{j} - \vu_k\bl{i} }\mA_{ji} + \big|e^{\vu_{k + 1}\bl{i} - \vu_{k}\bl{i}} - 1\big|\sum_{j = 1}^{i-1} e^{\vu_{k}\bl{i} - \vu_{k+1}\bl{j}}\mA_{ij} \Big) \\
        =\; & \sum_{i=2}^n \Big(\big|e^{\vu_{k}\bl{i} - \vu_{k + 1}\bl{i}} - 1\big| c_i^{1: i - 1}(\vu_{k, i}) + \big|e^{\vu_{k + 1}\bl{i} - \vu_{k}\bl{i}} - 1\big| r_i^{1: i - 1}(\vu_{k, i})\Big). 
    \end{align*}
    Now crudely bound $c_i^{1:i-1}(\vu_{k, i}) \leq c_i(\vu_{k, i})$, $r_i^{1:i-1}(\vu_{k, i}) \leq r_i(\vu_{k, i})$ and recall that $e^{\vu_{k + 1}\bl{i} - \vu_{k}\bl{i}} = \sqrt{c_i(\vu_{k, i})/ r_i(\vu_{k, i})}$ by the Osborne update~\eqref{eq:osborne-step}. Thus
    \begin{align*}
        \; &\|\nabla \varphi(\vu_{k+1})\|_1\\
        \leq \; & \sum_{i=2}^n\Big( \Big|\sqrt{r_i(\vu_{k, i})} - \sqrt{c_i(\vu_{k, i})}\Big|\sqrt{c_i(\vu_{k, i})} + \Big|\sqrt{r_i(\vu_{k, i})} - \sqrt{c_i(\vu_{k, i})}\Big|\sqrt{r_i(\vu_{k, i})}\Big)\\
        =\; &\sum_{i=2}^n |{r_i(\vu_{k, i})} - {c_i(\vu_{k, i})}|,
    \end{align*}
    completing the proof.
\end{proof}

Next we establish that Osborne's algorithm finds a ``good'' balancing every $K$ cycles, where ``good'' is quantified by how large $K$ is. The key insight is a way to combine the descent lemma (Lemma~\ref{lem:descent}) with the imbalance lemma (Lemma~\ref{lem:imbalance}). Recall that the latter lemma shows that the imbalance $\|\nabla \varphi(\vu_{k+1})\|_1$ at the end of a cycle $k$ is bounded by the imbalance of each row/column pair $j$ at the intermediate iterate $\vu_{k,j}$, right before it was updated in the cycle. At a conceptual level, we argue that either: $\Circled{1}$ all such imbalances are small, in which case the imbalance lemma ensures that $\vu_{k+1}$ is a good balancing; otherwise $\Circled{2}$ some row/column imbalances during the cycle were large, in which case Osborne's algorithm must have made significant progress when it corrected those imbalances, hence the objective function $\varphi$ drops substantially by the descent lemma. 
More precisely, the two cases are codified by comparing the $\ell_1$-type imbalance quantities $|r_j(\vu_{k, j}) - c_j(\vu_{k, j})| = |\sqrt{r_j(\vu_{k, j})} - \sqrt{c_j(\vu_{k, j})}| \cdot |\sqrt{r_j(\vu_{k, j})} + \sqrt{c_j(\vu_{k, j})}|$ in the imbalance lemma with the $\ell_{1/2}$-type imbalance quantities $(\sqrt{r_j(\vu_{k, j})} - \sqrt{c_j(\vu_{k, j})})^2$ in the descent lemma, namely by considering whether $|\sqrt{r_j(\vu_{k, j})} - \sqrt{c_j(\vu_{k, j})}|$ is much smaller than $\sqrt{r_j(\vu_{k, j})} + \sqrt{c_j(\vu_{k, j})}$ or not.

\begin{proposition}\label{prop:grad-bound}
For any $\ebar \in (0,1)$, number of cycles $K' \geq 1$, and initialization $\vu_0 \in \R^n$ (not necessarily zero), there exists $t \in [K']$ such that
\begin{align*}
    \frac{\|\nabla \varphi(\vu_{t})\|_1}{\varphi(\vu_{t})} \leq \;& 
    \ebar + \frac{10\varphi(\vu_0)}{\ebar K' \varphi(\vu_{K'})}.
\end{align*}
\end{proposition}
\begin{proof}
By an averaging argument, there exists $k \in \{0, 1, \dots, K' - 1\}$ satisfying
\begin{align}\label{eq:descent-k}
    \varphi(\vu_{k}) - \varphi(\vu_{k + 1})
    \leq \frac{1}{K'} \sum_{s = 0}^{K' - 1}\big(\varphi(\vu_{s}) - \varphi(\vu_{s + 1})\big)
    =
    \frac{\varphi(\vu_0) - \varphi(\vu_{K'})}{K'}\,.
\end{align}
We will show that the claimed bound holds for $t = k+1$. By Lemma~\ref{lem:imbalance},
\begin{align}\label{eq:grad-total}
     \|\nabla \varphi(\vu_{k+1})\|_1 \leq \sum_{j=2}^n \big|r_j(\vu_{k, j}) - c_j(\vu_{k, j})\big|.
\end{align}
To bound these summands, let $\alpha > 0$ be an analysis parameter (chosen below) and consider the following two cases of row/column imbalance for each coordinate $j \in [n]$:
\begin{enumerate}[label=\Circled{\arabic*},font=\sffamily]
    \item $\big|\sqrt{r_j(\vu_{k, j})} - \sqrt{c_j(\vu_{k, j})}\big| < \alpha \big(\sqrt{r_j(\vu_{k, j})} + \sqrt{c_j(\vu_{k, j})}\big),$ 
    \item $\big|\sqrt{r_j(\vu_{k, j})} - \sqrt{c_j(\vu_{k, j})}\big| \geq \alpha \big(\sqrt{r_j(\vu_{k, j})} + \sqrt{c_j(\vu_{k, j})}\big).$ 
\end{enumerate}
For the coordinates $j$ satisfying~\Circled{1}, we have
\begin{align*}
    \big|r_j(\vu_{k, j}) - c_j(\vu_{k, j})\big|
    &= 
    \Big|\Big( \sqrt{r_j(\vu_{k, j})} - \sqrt{c_j(\vu_{k, j})}\Big) \cdot \Big( \sqrt{r_j(\vu_{k, j})} + \sqrt{c_j(\vu_{k, j})}\Big) \Big|
    \\ &< \alpha \Big( \sqrt{r_j(\vu_{k, j})} + \sqrt{c_j(\vu_{k, j})}\Big)^2
    \\ &\leq 2\alpha \Big( r_j(\vu_{k, j}) + c_j(\vu_{k, j}) \Big)\,.
\end{align*}
Above, the first step is by factoring, the second step is by~\Circled{1}, and the third step is by the elementary inequality $(a+b)^2 \leq 2(a^2 + b^2)$. Now note that 
\begin{equation}\label{eq:inter-final-bound}
\begin{aligned}
    r_j(\vu_{k, j}) = \;& 
    e^{\vu_{k}\bl{j} - \vu_{k + 1}\bl{j}} r_j^{1: j - 1}(\vu_{k + 1}) + r_j^{j + 1: n}(\vu_{k}) \leq
    \sqrt{\frac{r_j(\vu_{k, j})}{c_j(\vu_{k, j})}}
    r_j(\vu_{k + 1}) + r_j(\vu_{k}), \\
    c_j(\vu_{k, j}) = \;& 
    e^{\vu_{k + 1}\bl{j} - \vu_{k}\bl{j}} c_j^{1: j - 1}(\vu_{k + 1}) + c_j^{j + 1: n}(\vu_{k}) \leq 
    \sqrt{\frac{c_j(\vu_{k, j})}{r_j(\vu_{k, j})}}
    c_j(\vu_{k+1})
    + c_j(\vu_{k}).
\end{aligned}
\end{equation}
Above, for both the row/columns, the first step is by splitting the sum (see~\eqref{eq:partial-row-column-sums-2}) and using the definition of the iterates (see last line of Algorithm~\ref{alg:ccd}), and the second step is by using the Osborne update~\eqref{eq:osborne-step} and crudely bounding partial row/column sums by full row/column sums. Now notice that~\Circled{1} is equivalent to
\begin{align*}
    \max \left\{ \sqrt{\frac{r_j(\vu_{k, j})}{c_j(\vu_{k, j})}}, \; \sqrt{\frac{c_j(\vu_{k, j})}{r_j(\vu_{k, j})}} \right\} < \frac{1+\alpha}{1 - \alpha}\,.
\end{align*}
Thus, by combining the above three displays,
\begin{align*}
    \Big|r_j(\vu_{k, j}) - c_j(\vu_{k, j})\Big|
    \leq
    2\alpha \Big[ r_j(\vu_{k}) + c_j(\vu_{k}) +  \frac{1+\alpha}{1-\alpha} \Big( r_j(\vu_{k+1}) + c_j(\vu_{k+1}) \Big) \Big]\,.
\end{align*}
Summing over such indices $j$, we can bound the terms in~\eqref{eq:grad-total} satisfying~\Circled{1} by
\begin{align}\label{eq:grad-1}
    &\;\sum_{j\,:\,\Circled{1}}\big|r_j(\vu_{k, j}) - c_j(\vu_{k, j})\big| \notag
    \\ \leq&\; 2\alpha \sum_{j\,:\,\Circled{1}} \Big[ r_j(\vu_{k}) + c_j(\vu_{k}) +  \frac{1+\alpha}{1-\alpha} \Big( r_j(\vu_{k+1}) + c_j(\vu_{k+1}) \Big) \Big]  \notag
    \\ \leq&\; 2\alpha \sum_{j=1}^n \Big[ r_j(\vu_{k}) + c_j(\vu_{k}) +  \frac{1+\alpha}{1-\alpha} \Big( r_j(\vu_{k+1}) + c_j(\vu_{k+1}) \Big) \Big]  \notag
    \\ =&\;  4 \alpha \Big[ \varphi(\vu_{k}) + \frac{1+\alpha}{1-\alpha} \,\varphi(\vu_{k+1}) \Big].
\end{align}

\par Next, for the coordinates $j$ satisfying~\Circled{2},
\begin{align*}
    \big|r_j(\vu_{k, j}) - c_j(\vu_{k, j})\big|
    &= 
    \Big|\Big( \sqrt{r_j(\vu_{k, j})} - \sqrt{c_j(\vu_{k, j})}\Big) \cdot \Big( \sqrt{r_j(\vu_{k, j})} + \sqrt{c_j(\vu_{k, j})}\Big) \Big|
    \\ &\leq \frac{1}{\alpha} \Big( \sqrt{r_j(\vu_{k, j})} - \sqrt{c_j(\vu_{k, j})}\Big)^2\,.
\end{align*}
Summing over such indices $j$ and applying Lemma~\ref{lem:descent}, we can bound the terms in~\eqref{eq:grad-total} satisfying~\Circled{2} by
\begin{align}\label{eq:grad-2}
    \sum_{j\,:\,\Circled{2}}\big|r_j(\vu_{k, j}) - c_j(\vu_{k, j})\big| \leq \;& \frac{1}{\alpha}\sum_{j\,:\,\Circled{2}}\Big(\sqrt{r_j(\vu_{k, j})} - \sqrt{c_j(\vu_{k, j})}\Big)^2 \notag \\
    \leq \;& \frac{1}{\alpha}\sum_{j = 1}^n \Big(\sqrt{r_j(\vu_{k, j})} - \sqrt{c_j(\vu_{k, j})}\Big)^2 \notag \\
    = \;& \frac{1}{\alpha} \big[ \varphi(\vu_{k}) - \varphi(\vu_{k+1})\big]\,.
\end{align}
By plugging~\eqref{eq:grad-1}~and~\eqref{eq:grad-2} into~\eqref{eq:grad-total}, and then re-arranging,
\begin{align*}
     \|\nabla \varphi(\vu_{k+1})\|_1 \leq \;& \sum_{j\,:\,\Circled{1}} \big|r_j(\vu_{k, j}) - c_j(\vu_{k, j})\big| + \sum_{j\,:\,\Circled{2}} \big|r_j(\vu_{k, j}) - c_j(\vu_{k, j})\big| \\
     \leq \;& 4\alpha \big[\varphi(\vu_{k}) + \frac{1+\alpha}{1-\alpha}\varphi(\vu_{k + 1})\big] + \frac{1}{\alpha} \big[ \varphi(\vu_{k}) - \varphi(\vu_{k+1}) \big] \\
     = \;& \frac{8\alpha}{1 - \alpha} \varphi(\vu_{k + 1})+ \Big(  \frac{1}{\alpha} + 4\alpha \Big) \big[ \varphi(\vu_{k}) - \varphi(\vu_{k+1}) \big]\,.
\end{align*}
Using~\eqref{eq:descent-k} to bound $\varphi(\vu_{k}) - \varphi(\vu_{k+1})$, setting $\alpha = \ebar / 9$, and crudely bounding the resulting constants gives
\[
 \|\nabla \varphi(\vu_{k + 1})\|_1 \leq 
    \ebar \varphi(\vu_{k + 1})+ \frac{10\big(\varphi(\vu_0) - \varphi(\vu_{K'})\big)}{\ebar K'}.
\]
The proof is complete by dividing both sides by $\varphi(\vu_{k+1})$, using Lemma~\ref{lem:descent} to bound $\varphi(\vu_{k+1}) \geq \varphi(\vu_{K'})$, and using the definition of $\varphi$ to crudely bound $\varphi(\vu_{K'}) \geq 0$.
\end{proof}

With Proposition~\ref{prop:grad-bound} in hand, we are now ready to establish a near-linear runtime for Osborne's algorithm. Recall that our main result (Theorem~\ref{thm:main}) is the minimum of two such bounds: one that scales as $1/\eps^2$ and the other that accelerates to $\d/\eps$. We first prove the former bound, and then below show how this automatically implies acceleration when $\d \leq 1/\eps$.

\begin{proposition}[Weak version of Theorem~\ref{thm:main}]\label{prop:weak}
    Given a balanceable matrix $\mA$, accuracy $\epsilon > 0$, and initialization $\vu_0$ (not necessarily zero), Osborne's algorithm (Algorithm~\ref{alg:ccd}) finds an $\eps$-balancing after $\cO\big(\frac{\log_2 (\varphi(\vu_0) / \varphi(\vu_*))}{\eps^2}\big)$ cycles. 
\end{proposition}
\begin{proof}
Applying Proposition~\ref{prop:grad-bound} with $\ebar = \frac{\eps}{2}$ and $K' = \frac{80}{\eps^2}$ ensures the existence of $k \in [K']$ such that
\begin{align*}
    \frac{\|\nabla \varphi(\vu_{k})\|_1}{\varphi(\vu_{k})} \leq \;& 
        \frac{\epsilon }{2} \left( 1 + \frac{\varphi(\vu_0)}{2\varphi(\vu_{K'})} \right)\,.
\end{align*}
If $\varphi(\vu_{K'}) \geq \varphi(\vu_0) / 2$, then the right hand side is at most $\eps$, thus by definition $\vu_{k}$ provides an $\eps$-balancing. Otherwise,
Osborne's algorithm halves the objective value in the first $K'$ cycles since $ \varphi(\vu_{K'}) < \varphi(\vu_0) / 2$. In that case, we can repeat the argument, now applying Proposition~\ref{prop:grad-bound} 
starting from $\vu_{K'}$, 
to ensure that either Osborne's algorithm finds an $\eps$-balancing in the next $K'$ cycles or halves the function value again, i.e., $\varphi(\vu_{2K'}) \leq \varphi(\vu_0)/4$. Note that the objective value can halve at most $\lceil \log_2 (\varphi(\vu_0) / \varphi(\vu_*)) \rceil $ times, otherwise  the objective value for Osborne's algorithm would be lower than $\varphi(\vu_*)$, which cannot happen. Therefore Osborne's algorithm finds an $\eps$-balancing after at most $\lceil \log_2 (\varphi(\vu_0) / \varphi(\vu_*)) \rceil \cdot K' = \cO(\tfrac{\log_2 (\varphi(\vu_0) / \varphi(\vu_*))}{\epsilon^2})$ cycles, as desired.
\end{proof}

We now upgrade the $\cO(1/\eps^2)$ convergence result in Proposition~\ref{prop:weak} to our main result. This upgrade conceptually follows the 2-phase argument in~\citep{altschuler2023near}, but operationally is different because there one could directly use the progress of each iteration of Osborne's algorithm, whereas here Proposition~\ref{prop:weak} only controls the progress of Osborne's algorithm indirectly through the existence of a good balancing within every batch of cycles. For the convenience of the reader, we restate Theorem~\ref{thm:main} here before proving it. 

\mainThm*

\begin{proof}
We prove the bound on the number of cycles; this  immediately implies the bound on the number of arithmetic operations since Osborne's algorithm uses $\cO(m)$ arithmetic operations per cycle. The bound on the number of cycles is the minimum of two bounds. The bound $\cO(\tfrac{\log \kappa}{\epsilon^2})$ follows from combining Proposition~\ref{prop:weak} with Lemma~\ref{lemma:bound-gap}; below we prove the other bound $\cO(\tfrac{\d \log \kappa }{\epsilon}$).
\par For shorthand, denote $a := \epsilon \d \log \kappa$ and $\Phi(\vu) := \log (\varphi(\vu) / \varphi(\vu_*))$. 
Let $K'$ denote the first cycle for which $\Phi(\vu_{K'}) \leq a$, and let $K''$ denote the number of remaining cycles. By Proposition~\ref{prop:weak},
\begin{align*}
    K'' = \cO\Big( \frac{a}{\epsilon^2} \Big) = \cO\Big( \frac{\d \log \kappa}{\epsilon} \Big)
\end{align*}
is of the desired order, so it now suffices to show that is also true for $K'$. 

To this end, let $a_0 := \log \kappa$ and $a_i := a_{i-1}/2$ for $i = 1, 2, \dots$ until $a_N \leq a$. Let $K_i'$ denote the number of cycles starting from when the potential $\Phi$ is no greater than $a_{i-1}$ and ending when it is no greater than $a_i$. We argue that
\[
K_i' = \cO\Big( \frac{a_i}{\eps_i^2} \Big) \qquad \text{ where } \qquad  \eps_i := \frac{a_i}{\d \log \kappa}\,.
\]
By Proposition~\ref{prop:weak}, there exists an $\eps_i$-balanced iterate $\vu_{k_i}$ within the $K_i'$ cycles:
\[
    \nabla \Phi(\vu_{k_i}) = \frac{\|\nabla \varphi(\vu_{k_i})\|_1}{\varphi(\vu_{k_i})}\leq \eps_i\,.
\]
Also, by Lemma~\ref{lemma:bound-dist}, there exists a minimizer $\vu_*$ satisfying $\|\vu_{k_i} - \vu_*\|_\infty  \leq \d \log \kappa$. Thus by using the fact that $\Phi(\vu_*) = 0$, convexity of $\Phi$ (see Lemma~\ref{lem:prelim}), and H\"older's inequality,
\begin{align*}
    \Phi(\vu_{k_i}) = \Phi(\vu_{k_i}) - \Phi(\vu_*)
    \leq \|\nabla \Phi(\vu_{k_i}) \|_1 \|\vu_{k_i} - \vu_*\|_\infty 
    \leq \epsilon_i \d \log \kappa = a_i\,.
\end{align*}
Hence $\vu_{k_i}$ would have already met the threshold to exit the $i$-th phase, and by monotonicity of Osborne's algorithm (Lemma~\ref{lem:descent}) this implies the claimed bound on $K_i'$. By bounding $\sum_{i=1}^N \frac{1}{a_i} = \frac{1}{a_N} \sum_{j=0}^{N-1} 2^{-j} \leq \frac{2}{a_N} \leq \frac{4}{a}$, we conclude that $K'$ is of the desired order
\[
    K'
    = \sum_{i=1}^N K_i'
    = \cO\Big( \d^2 \log^2 \kappa \sum_{i=1}^N \frac{1}{a_i}\Big)
    = \cO\Big( \frac{\d^2 \log^2 \kappa}{a}\Big)
    = \cO\Big( \frac{\d \log \kappa}{\epsilon}\Big)\,.
\]
\end{proof}

\subsection{Discussion and extensions}\label{ssec:analysis:discussion}

A few remarks are in order. First, as mentioned before, Theorem~\ref{thm:main} establishes the first near-linear runtime for Osborne's algorithm, improving by a factor of $\cO(n^2)$ over the prior best~\cite{ostrovsky2017matrix}. 
Additionally, our runtime matches that of random Osborne~\citep{altschuler2023near} (which achieved the best-known rate for any variant of Osborne's algorithm) and moreover improves in the sense that our bound is deterministic whereas the bound from \cite{altschuler2023near} only applies in expectation or with high probability.

An interesting aspect of this result is that it provides a problem instance  where a cyclic coordinate method shares the same complexity bound as its randomized counterpart. Note that for quadratic minimization problems, cyclic coordinate descent is known to be order-$n^2$ times slower than randomized coordinate descent in the worst case \cite{sun2021worst}. More generally, existing complexity bounds for cyclic methods are almost always worse than even the traditional full-vector analogs (i.e., single block methods), by polynomial factors in $n$ (e.g., $\sqrt{n}$ or $n^{1/4}$) in the worst case \cite{beck2013convergence,song2021coder,lin2023accelerated,cai2022cyclic,gurbuzbalaban2017cyclic,wright2020analyzing}. In other words, not only is the worst-case complexity of cyclic methods higher than the complexity of randomized block coordinate methods, but even higher than the complexity of traditional (single-block) methods over which randomized methods can improve by factors scaling with $\sqrt{n}$ or $n.$ The only exception in the literature is \cite{chakrabarti2023block}; there, the complexity of a cyclic method is of the same order as for traditional (single block) methods, i.e., for that problem, using a block coordinate method does not appear to lead to theoretical complexity improvements. 

Additionally, an appealing aspect of our analysis is that it is does not rely on the ordering of updates within a cycle, making it applicable to arbitrary (and potentially changing) orderings across cycles. This has several implications. One is a (deterministic) near-linear runtime  which improves by an $\cO(n)$ factor over existing bounds for the random-reshuffle variant of Osborne's algorithm~\cite{altschuler2023near}.
\begin{corollary}[Shuffled Osborne]\label{cor:permutation}
Given a balanceable matrix $\mA$ and target accuracy $\epsilon > 0$, cyclic Osborne with arbitrary (and potentially changing) ordering of updates across cycles finds an $\epsilon$-balancing
after $K = \cO(\tfrac{\log \kappa}{\epsilon}\min\{\tfrac{1}{\epsilon}, \d\})$ cycles. This amounts to   $\cO(\tfrac{m\log \kappa}{\epsilon}\min\{\tfrac{1}{\epsilon}, \d\})$ arithmetic operations.   
\end{corollary}

Another implication is for parallelizing Osborne's algorithm. It is well-known that Osborne's algorithm can be parallelized without any change to its output, when given a coloring of $G_{\mA}$, i.e., a partition $S_1, \dots, S_p$ of the vertices $[n]$ such that for any partition $S_{\ell}$, all vertices are non-adjacent~\citep{bertsekas2015parallel}. The idea is that since any vertices $i,j$ in the same partition $S_{\ell}$ are non-adjacent, their updates in Osborne's algorithm do not affect each other, hence updating them concurrently is equivalent to updating them sequentially. See~\citep{altschuler2023near} for further discussion of how this parallelized variant of Osborne's algorithm corresponds to \emph{block} exact coordinate descent, and how one can leverage sparsity structure of $G_{\mA}$ to find small colorings, e.g., if the max-degree $G_{\mA}$ is bounded (i.e., if each row and column of $\mA$ has a bounded number of non-zero entries).

\begin{corollary}[Parallelized Osborne]\label{cor:parallelized}
    Consider the setup of Theorem~\ref{thm:main}, and suppose also that one is given a coloring $S_1, \dots, S_p$ of $G_{\mA}$. Then the same runtime bound of $K = \cO(\tfrac{\log \kappa}{\epsilon} \min\{ \tfrac{1}{\epsilon},\d\})$ cycles applies to parallelized Osborne. This amounts to the same amount of total work $\cO(mK)$, but only $\cO(pK)$ rounds of parallel computation. 
\end{corollary}
\begin{proof}
    Since our analysis applies for any ordering of the $n$ update indices, it applies in particular if one first updates all entries in $S_1$, then all entries in $S_2$, and so on. Since the updates in each partition do not affect other, it is equivalent to update the coordinates in each partition in sequential order (as in our original analysis), or to do them all at once (the parallelized Osborne algorithm). Hence the proof of Theorem~\ref{thm:main} directly extends to give the same bound on the number of cycles $K$. This constitutes the same amount of work $\cO(mK)$, but only requires $\cO(pK)$ synchronization rounds of parallel computation. 
\end{proof}

Another appealing feature of our analysis is that it holds even when exact arithmetic is replaced by arithmetic on numbers with low bit complexity: the same near-linear runtime bound (up to universal constants) holds on the number of arithmetic operations when using only $\cO\big(\log(\frac{n\kappa }{\epsilon}) 
\big)
$-bit numbers\footnote{The bit complexity further improves to $\cO(\log\tfrac{n}{\kappa} + \log \log \tfrac{\mA_{\mathrm{max}}}{\mA_{\mathrm{min}}})$, where $\mA_{\mathrm{max}}$ and $\mA_{\mathrm{min}}$ denote the maximum and minimum non-zero entries of $\mA$, if $\mA$ is input through the logarithms of its non-zero entries. This is the case in some applications, e.g., for computing the minimum-mean-cycle of a graph~\citep{altschuler2022approximating}.}. This is an exponential improvement over the $\cO(n \log \kappa)$ bit complexity required by the previous state-of-the-art for Osborne's original algorithm~\citep{ostrovsky2017matrix}, and constitutes an additional $\tilde{\cO}(n)$ improvement in runtime beyond the aforementioned $\cO(mn)$ improvement in the number of arithmetic operations. Our implementation is based on the one in~\cite[Section 8]{altschuler2023near} for randomized variants of Osborne's updates but requires additional gadgets and a more sophisticated analysis: whereas those arguments could simply rely on additive monotonic decrease of $\log \varphi$ and argue that truncation contributes a smaller additive error~\cite{altschuler2017near}, this is insufficient for our analysis of cyclic Osborne updates. Instead, our argument is based on two claims: (1) If cyclic Osborne fails to decrease the objective value in a cycle, then the current iterate is $\epsilon$-balanced. (2) Otherwise, the truncation error can be absorbed in bounding $\|\nabla \varphi(\vu_{k+1})\|_1$, leading to qualitatively the same key bound as in Proposition \ref{prop:grad-bound}, which then directly leads to the same runtime as in the exact arithmetic setting above. We state the theorem below and defer the proof to Appendix~\ref{appx:bit} for brevity.

\begin{restatable}[Bit Complexity]{theorem}{bitComplexity}\label{thm:bit}
Given a balanceable matrix $\mA$ and accuracy $\epsilon > 0$, Osborne's algorithm with arbitrary (and potentially changing) orderings of updates across cycles finds an $\epsilon$ balancing 
after $\cO\big(\frac{m\log \kappa}{\epsilon}\min\{\frac{1}{\epsilon}, \d\}\big)$ arithmetic operations over $\cO\big(\log \frac{n \kappa}{\epsilon}
\big)
$-bit numbers.    
\end{restatable}

\section{Illustrative Numerical Examples}\label{sec:num-ex}

As discussed earlier in the paper, the main motivation for this paper is that Osborne's original 1960 algorithm for matrix balancing has remained the practical state of the art, used by default (e.g., before eigenvalue computation) in mainstream software packages such as Python, Julia, MATLAB, LAPACK, and EISPACK. As discussed in the introduction, recent work has proposed variants of Osborne's algorithm that choose update indices in different ways because this makes the algorithm simpler to analyze. Although the focus of this paper is theoretical, this section provides illustrative numerical results comparing all these variants, in order to highlight why Osborne's original algorithm has remained the default implementation.

\par Specifically, we compare Osborne's original algorithm (which simply updates coordinates $i \in [n]$ in sorted cyclic order), greedy Osborne~\citep{ostrovsky2017matrix} (which chooses update coordinates $i$ with maximal imbalance $(\sqrt{r_i} - \sqrt{c_i})^2$), weighted random Osborne~\citep{ostrovsky2017matrix} (which chooses update coordinates $i \in [n]$ with probability proportional to $r_i + c_i$), random Osborne~\citep{altschuler2023near} (which chooses update coordinates $i$ uniformly at random from $[n]$), and random-reshuffle Osborne~\citep{altschuler2023near} (which randomly shuffles the update coordinates for each cycle). We compare these algorithms on two families of inputs; qualitatively similar results are obtained on other inputs. In the first setup, we generate a random matrix $\mA \in \R^{1000 \times 1000}$ with a small number of salient rows and columns, where all entries are uniformly drawn from the interval $(0, 0.001)$ except for the entries in the last $20$ rows and columns, which are uniformly drawn from the interval $(0, 1)$. The rationale for this input instance is that one might expect cyclic Osborne to perform poorly since it spends equal time updating all rows and columns, not just the salient ones. 
In the second setup, we use the hard  instance from~\cite{kalantari1997complexity} where $\mA_{i, i + 1} = \mA_{2k + 2 - i, 2k + 1 - i} = 1$, $\mA_{i + 1, i} = \mA_{2k + 1 - i, 2k + 2 - i} = 0.01$ for $i \in [k]$, and $\mA_{n, 1} = \mA_{1, n} = 1$, with $k = 40$ and $n = 2k + 1$. The rationale for this  instance is that it requires finding diagonal scalings with extreme entries, see~\citep[\S3]{kalantari1997complexity}.
Figure~\ref{fig:exp} plots the $\ell_1$ imbalance of the algorithms against the number of iterations (left), the number of nonzero entries the algorithm touches (middle), and the wall-clock time (right), until error $10^{-10}$ is reached. All experiments are run in MATLAB on a  MacBook Pro laptop.

\begin{figure}[t]
    \hspace*{\fill}
    \subfloat[Number of iterations]{\includegraphics[width=0.325\textwidth]{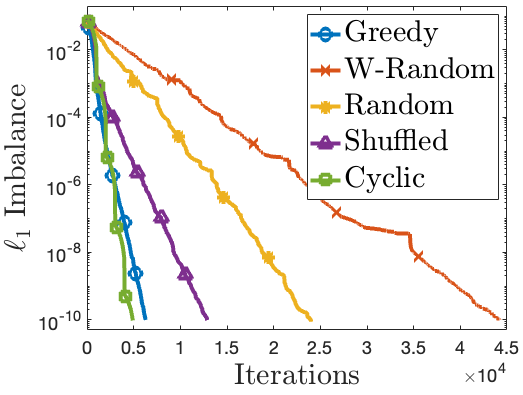}\label{fig:random-iter}}
    \hfill
    \subfloat[Number of nonzeros]{\includegraphics[width=0.325\textwidth]{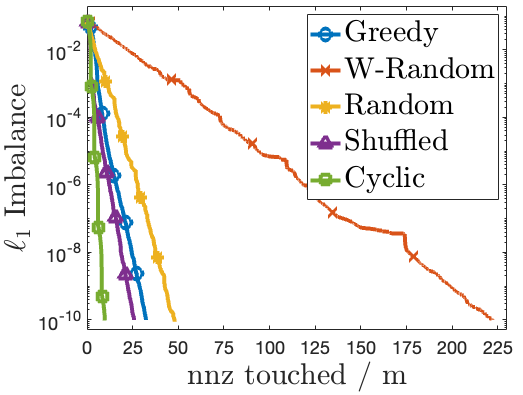}\label{fig:random-nnz}}
    \hfill
    \subfloat[Wall-clock time]{\includegraphics[width=0.325\textwidth]{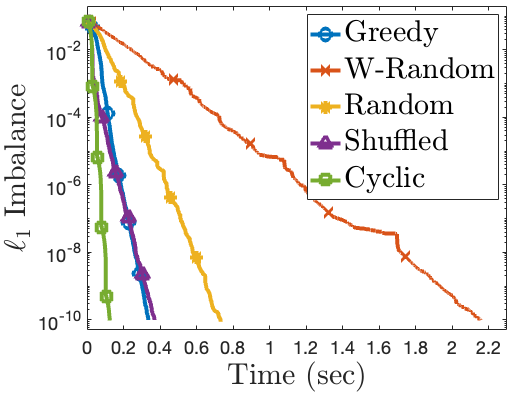}\label{fig:random-time}}
    \hspace*{\fill} \\
    \hspace*{\fill}
    \subfloat[Number of iterations]{\includegraphics[width=0.325\textwidth]{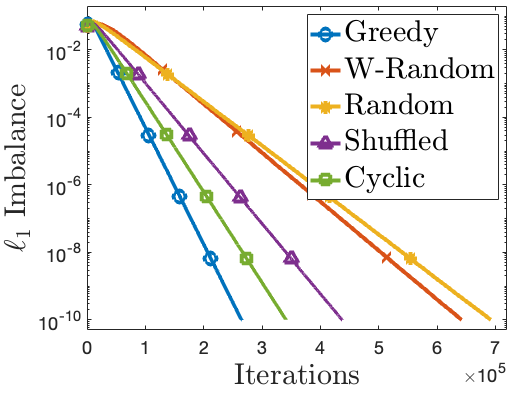}\label{fig:KKO-iter}}
    \hfill
    \subfloat[Number of nonzeros]{\includegraphics[width=0.325\textwidth]{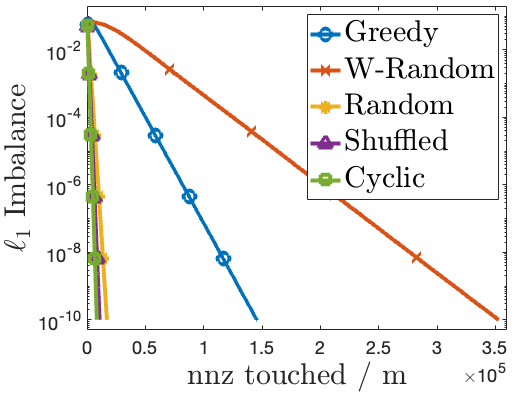}\label{fig:KKO-nnz}}
    \hfill
    \subfloat[Wall-clock time]{\includegraphics[width=0.325\textwidth]{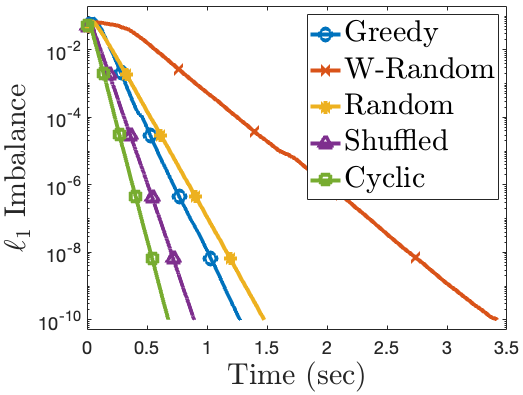}\label{fig:KKO-time}}
    \hspace*{\fill}
    \caption{Comparison of different variants of Osborne's algorithm for matrix balancing. First row: random input with a small number of salient rows/columns. Second row: hard input from~\cite{kalantari1997complexity}. Each row has three plots, with different comparison criteria on the horizontal axis. Osborne's original algorithm (cyclic) consistently outperforms other variants, which is why it is the standard implementation.}
    \label{fig:exp}
\end{figure}

Observe that on both problem instances, Osborne's original (cyclic) algorithm  outperforms all other variants in terms of the number of nonzeros touched and wall-clock time. Notice that it converges much faster than random variants of Osborne, despite their theoretical complexity bounds being of the same asymptotic order~\citep{altschuler2023near}. Notice also that---in contrast to cyclic coordinate descent on convex quadratic problems~\citep{wright2020analyzing}---randomly reshuffling the ordering does not provide additional runtime benefits, while it introduces extra overhead for generating the random permutation. Weighted random Osborne is slowed down on both problems by the overhead of updating and maintaining the sampling probabilities for each row/column.
Finally, we note that greedy Osborne is comparable to cyclic Osborne in terms of the number of update iterations, since both problems are dense and their runtime bounds are of the same asymptotic order. However, greedy Osborne is much slower when it comes to the number of nonzeros touched and wall-clock time, due to the overhead for greedy coordinate selection (an extra data structure is required for amortizing the cost of greedy selection) and the fact that cyclic Osborne is cache-friendly as it accesses adjacent rows and columns in consecutive iterations.

\section{Conclusion}

This paper provides the first near-linear runtime bound for Osborne's classical algorithm for matrix balancing \cite{osborne1960pre}. This not only improves over state-of-the-art results by large polynomial factors in the problem dimension, but also provides the first runtime that is not dominated by the runtime of downstream tasks like eigenvalue computation, explaining the longstanding empirical observation that Osborne's algorithm preconditions in an essentially negligible amount of time compared to downstream tasks~\citep[\S11.6.1]{press2007numerical}. While multiple variants of Osborne's iteration have been considered in the recent literature, our results apply to Osborne's original proposal of cyclic (or round-robin) update indices, which has long remained the practical state-of-the-art. Our result proves it is also theoretical state-of-the-art, effectively closing this gap between theory and practice, and answering open questions in~\cite{schulman2017analysis,altschuler2023near}.

\section*{Acknowledgments}
This work was supported in part by the Air Force Office of Scientific Research under award number FA9550-24-1-0076, by the U.S.\ Office of Naval Research under contract number N00014-22-1-2348, and by a Seed Grant Award from Apple. Any opinions, findings and conclusions or recommendations expressed in this material are those of the author(s) and do not necessarily reflect the views of the U.S. Department of Defense.

\bibliographystyle{plainnat}
\bibliography{reference}

\appendix

\section{Bit Complexity}\label{appx:bit}

This section provides a proof of Theorem \ref{thm:bit}. The implementation is based on the one in~\citep{altschuler2023near} for randomized variants of Osborne's algorithm, but requires a more sophisticated analysis for Osborne's original cyclic algorithm; see the discussion in \S\ref{ssec:analysis:discussion}.

\begin{implementation}\label{alg:implementation}
Given a target accuracy $\epsilon \in (0, 1)$:
\begin{enumerate}[label=(\roman*)]
    \item Preprocess $\mA$ by computing $\big\{\log\mA_{ij} \big\}_{(i, j): \mA_{ij} \neq 0}$ to additive accuracy $\gamma = \Theta(\epsilon / n)$.
    \item Truncate Osborne iterates entrywise to additive accuracy $\gamma' = \Theta(\epsilon^2)$.
    \item Compute Osborne iterates entrywise to additive accuracy $\gamma'$ using log-sum-exp computation tricks 
    (Lemma 17 of~\citep{altschuler2023near})
    and accessing $\mA_{ij}$ only through the truncated $\log \mA_{ij}$ values computed in $(\romannumeral1)$. 
    \item After the completion of every cycle, check whether the current iterate is $\bar{\epsilon}$-balanced for $\bar{\epsilon} = \Theta(\epsilon)$ using Lemma~\ref{lem:termination-err}.
\end{enumerate}
\end{implementation}
We denote the iterates of this inexact implementation of Osborne's algorithm by $\{\bar \vu_{k, j}\}$, to distinguish them from the iterates $\{\vu_{k, j}\}$ of the exact implementation in Algorithm~\ref{alg:ccd}. We argue that none of the four inexact implementation steps affects the analysis by more than a universal constant. 
\par By~\citep[Lemma 14]{altschuler2023near}, the truncation in $(\romannumeral1)$ is precise enough that an $\cO(\epsilon)$-balancing of the truncated version of $A$ is an $\cO(\epsilon)$-balancing of $A$; therefore we henceforth ignore $(\romannumeral1)$ by supposing $A$ is not truncated. 
\par The verification of the termination criteria $(\romannumeral4)$ is handled by the following simple lemma, which shows that using low bit complexity, the balancing criterion $g := \|\nabla \varphi\|/\varphi$ can be estimated to sufficiently high accuracy that one can terminate if and only if $g = \Theta(\bar \eps)$.

\begin{lemma}[Inexact verification of termination criteria]\label{lem:termination-err}
    Suppose that $\bar{\vu}$ and $\{\log \mA_{ij}\}_{ij : \mA_{ij} \neq 0}$ are represented using $b$ bits, and let $\bar \eps < 1$. Let $g := \|\nabla \varphi(\bar \vu)\|_1/\varphi(\bar \vu)$. Using $\cO(m)$ operations on $\cO(b + \log \frac{n}{\bar \epsilon})$-bit numbers, one can compute $\hat g$ such that $\tfrac{g}{2} - \tfrac{\bar \eps}{2} \leq \hat g \leq 2g + \tfrac{\bar \eps}{2}$.
\end{lemma}

\begin{proof}
    For shorthand, denote $ r_j := r_j(\bar \vu)$, $c_j := c_j(\bar \vu)$, and $\varphi := \varphi(\bar \vu)$. Then $g = \sum_{j=1}^n |r_j - c_j| / \varphi$. By the log-sum-exp trick in~\citep[Lemma 17]{altschuler2023near}, using $\cO(m)$ operations on $\cO(b + \log \frac{n}{\rho})$-bit numbers, we can compute $\hat{r}_j$, $\hat{c}_j$, $\hat{\varphi}$ that entrywise are $e^{\pm \rho}$ multiplicative approximations of $r_j$, $c_j$, $\varphi$, respectively. Now by the reverse triangle inequality, the multiplicative approximation guarantee, and the definition of $\varphi$,
    \begin{align*}
        \sum_{j=1}^n \Big| |r_j - c_j| - |\hat{r}_j - \hat{c}_j| \Big| 
        \leq 
        \sum_{j=1}^n \Big( |r_j - \hat{r}_j| + |c_j - \hat{c}_j| \Big)
        \leq 
        (e^{\rho} - 1) \sum_{j=1}^n \Big( r_j + c_j \Big)
        = 
        2(e^{\rho} - 1) \varphi\,.
    \end{align*}
    It follows that the estimate $\hat g := \sum_{j=1}^n |\hat{r}_j - \hat{c}_j| / \hat{\varphi}$ satisfies
    \begin{align*}
        \hat g 
        = 
        \frac{\sum_{j=1}^n |\hat{r}_j - \hat{c}_j|}{\hat{\varphi}}
        \leq 
        \frac{\sum_{j=1}^n |r_j - c_j| + 2(e^{\rho}-1)\varphi}{e^{-\rho} \varphi}
        = e^{\rho} g + 2e^{\rho} (e^{\rho} - 1)\,,
    \end{align*}
    and similarly
    \begin{align*}
        \hat g 
        =
        \frac{\sum_{j=1}^n |\hat{r}_j - \hat{c}_j|}{\hat{\varphi}}
        \geq 
        \frac{\sum_{j=1}^n |r_j - c_j|}{e^{\rho} \varphi} - \frac{2(e^{\rho}-1)\varphi}{e^{-\rho} \varphi}
        = e^{-\rho} g - 2e^{\rho} (e^{\rho} - 1)\,.
    \end{align*}
    Setting $\rho = \Theta(\bar \epsilon)$ with an appropriate choice of constant completes the proof.
\end{proof}

\par It therefore remains only to control the error due to $(\romannumeral2)$ and $(\romannumeral3)$. By~\citep[Lemmas 15 and 16]{altschuler2023near}, the combined effect of $(\romannumeral2)$ and $(\romannumeral3)$ is that
\begin{align}
    \left| \vu_{k + 1}\bl{j} - \left( \vu_{k}\bl{j} + \frac{1}{2} \log c_j(\bar \vu_{k,j}) - \frac{1}{2} \log r_j(\bar \vu_{k,j}) \right) \right| \leq 2\gamma'\,,
\end{align}
which is an inexact version of the update in Osborne's algorithm (c.f.\,Algorithm~\ref{alg:ccd}) up to $ 2\gamma'$ additive error. By exponentiating, bounding $|1 - e^{2\gamma'}| = \cO(\gamma')$, and substituting $\tau = \cO(\gamma')$ for an appropriate constant,
\begin{equation}\label{eq:update-multi}
\begin{aligned}
    \;& \Big| e^{\bar \vu_{k + 1}\bl{j} - \bar \vu_{k}\bl{j}} - \sqrt{c_j(\bar \vu_{k,j}) / r_j(\bar \vu_{k, j})} \Big| \leq \tau\sqrt{ c_j(\bar \vu_{k,j}) / r_j(\bar \vu_{k, j})}, \\
    \;& \Big| e^{\bar \vu_{k}\bl{j} - \bar \vu_{k + 1}\bl{j}} - \sqrt{ r_j(\bar \vu_{k, j}) / c_j(\bar \vu_{k,j})} \Big| \leq \tau \sqrt{ r_j(\bar \vu_{k, j}) / c_j(\bar \vu_{k,j})}.
\end{aligned}
\end{equation}
As a result, this inexact Osborne update only balances the $j\textsuperscript{th}$ row/column sums approximately, with error that scales with $\sqrt{r_j(\bar \vu_{k, j}) c_j(\bar \vu_{k, j})}.$
\begin{fact}[Inexact version of Fact~\ref{lem:step}]\label{lem:step-err}
    For every $k \geq 0$ and $j \in [n],$
     \begin{align*}
         \;& \Big| r_j(\bar \vu_{k, j + 1}) - \sqrt{r_j(\bar \vu_{k, j}) c_j(\bar \vu_{k, j})}\Big| \leq \tau \sqrt{r_j(\bar \vu_{k, j}) c_j(\bar \vu_{k, j})}, \\
         \;& \Big| c_j(\bar \vu_{k, j + 1}) - \sqrt{r_j(\bar \vu_{k, j}) c_j(\bar \vu_{k, j})} \Big| \leq \tau \sqrt{r_j(\bar \vu_{k, j}) c_j(\bar \vu_{k, j})}.
     \end{align*}
     As a consequence, we also have 
     \begin{align*}
         \big|  r_j(\bar \vu_{k, j + 1}) -  c_j(\bar \vu_{k, j + 1})\big| \leq 2\tau\sqrt{r_j(\bar \vu_{k, j}) c_j(\bar \vu_{k, j})}.
     \end{align*}
\end{fact}

We bound the accumulation of this imbalancing error by appropriately modifying the analysis in \S\ref{sec:analysis}. We begin by bounding the error in the descent lemma (Lemma~\ref{lem:descent}) and the imbalance lemma (Lemma~\ref{lem:imbalance}), showing that both hold up to an error term of the same order $\tau \sum_{j=1}^n \sqrt{r_j(\bar \vu_{k, j})c_j(\bar \vu_{k, j})}$.
\begin{lemma}[Inexact version of Lemma~\ref{lem:descent}]\label{lem:descent-err}
    For all $k \geq 0$ and $j \in [n]$,
    \begin{equation}\notag
    \begin{aligned}
        \varphi(\bar \vu_{k, j + 1}) - \varphi(\bar \vu_{k, j}) 
        \leq - \Big(\sqrt{r_j(\bar \vu_{k, j})} - \sqrt{c_j(\bar \vu_{k, j})}\Big)^2 + 2\tau\sqrt{r_j(\bar \vu_{k, j})c_j(\bar \vu_{k, j})}. 
    \end{aligned}
    \end{equation}
    Hence, for all $k \geq 0,$ the descent of the potential over the full cycle is
    \begin{equation}\notag
    \begin{aligned}
        \varphi(\bar \vu_{k+1}) - \varphi(\bar \vu_k) 
        \leq - \sum_{j = 1}^n \Big(\sqrt{r_j(\bar \vu_{k, j})} - \sqrt{c_j(\bar \vu_{k, j})}\Big)^2 + 2\tau\sum_{j = 1}^n\sqrt{r_j(\bar \vu_{k, j})c_j(\bar \vu_{k, j})}. 
    \end{aligned}
    \end{equation}
\end{lemma}
\begin{proof}
    Identical to the proof of Lemma~\ref{lem:descent}, except replace Fact~\ref{lem:step} by its inexact version in Fact~\ref{lem:step-err}.
\end{proof}

\begin{lemma}[Inexact version of Lemma~\ref{lem:imbalance}]\label{lem:imbalance-err}
     For every $k \geq 0$,
     \begin{equation}
        \|\nabla \varphi(\bar \vu_{k+1})\|_1 
        \leq \sum_{j=2}^n \big|r_j(\bar \vu_{k, j}) - c_j(\bar \vu_{k, j})\big| + 4\tau \sum_{j=1}^n \sqrt{r_j(\bar \vu_{k, j})c_j(\bar \vu_{k, j})}. 
    \end{equation}
\end{lemma}
\begin{proof} We largely follow the proof of Lemma~\ref{lem:imbalance}, with additional care to bound the accumulated effect of the inexact updates.
We begin with an inexact version of~\eqref{eq:imbalance-decompose}: for any $j \in [n]$, the $j\textsuperscript{th}$ coordinate of the gradient $\nabla \varphi$ after completion of cycle $k$ is equal to:
    \begin{align*}
       \nabla_j \varphi(\bar \vu_{k+1}) = \;& r_j^{1: j - 1}(\bar \vu_{k+1}) + r_j^{j + 1: n}(\bar \vu_{k+1}) - c_j^{1: j - 1}(\bar \vu_{k+1}) - c_j^{j + 1: n}(\bar \vu_{k+1}) \\
       = \; & \Big(r_j^{j+1:n}(\bar \vu_{k+1}) - r_j^{j+1:n}(\bar \vu_{k, j + 1})\Big) - \Big(c_j^{j+1:n}(\bar \vu_{k+1}) - c_j^{j+1:n}(\bar \vu_{k, j + 1})\Big) \\
       & + \Big(r_j(\bar \vu_{k, j + 1}) - c_j(\bar \vu_{k, j + 1})\Big),
    \end{align*}
    where the argument is the same as in~\eqref{eq:imbalance-decompose}, except that in the final step we expand $r_j^{1: j - 1}(\bar \vu_{k+1}) = r_j^{1: j - 1}(\bar \vu_{k, j+1}) = r_j(\bar \vu_{k, j+1}) - r_j^{j+1: n}(\bar \vu_{k, j+1})$ and similarly $c_j^{1: j - 1}(\bar \vu_{k+1}) = c_j(\bar \vu_{k, j+1}) - c_j^{j+1: n}(\bar \vu_{k, j+1})$. Note that the above display matches~\eqref{eq:imbalance-decompose} up to the error term $r_j(\bar \vu_{k, j + 1}) - c_j(\bar \vu_{k, j + 1})$. Thus, by following the proof of Lemma~\ref{lem:imbalance} to bound the matching term and by using Fact~\ref{lem:step-err} to bound the additional error term,
     \begin{align*}
        \|\nabla \varphi(\bar \vu_{k+1})\|_1 
        \leq \;& \sum_{i=2}^n \Big(\big|e^{\bar \vu_{k}\bl{i} - \bar \vu_{k + 1}\bl{i}} - 1\big| c_i(\bar \vu_{k, i}) + \big|e^{\bar \vu_{k + 1}\bl{i} - \bar \vu_{k}\bl{i}} - 1\big| r_i(\bar \vu_{k, i})\Big) + 2\tau \sum_{i=1}^n \sqrt{r_i(\bar \vu_{k, i})c_i(\bar \vu_{k, i})}. 
    \end{align*}
    By the triangle inequality and the bound~\eqref{eq:update-multi} on the inexact Osborne update, 
    \begin{align*}
        \big|e^{\bar \vu_{k}\bl{i} - \bar \vu_{k + 1}\bl{i}} - 1 \big| c_i(\bar \vu_{k, i}) \leq \;& \Bigg|e^{\bar \vu_{k}\bl{i} - \bar \vu_{k + 1}\bl{i}} - \sqrt{\frac{r_i(\bar \vu_{k, i})}{c_i(\bar \vu_{k,i})}} \Bigg| c_i(\bar \vu_{k, i}) + \Bigg| \sqrt{\frac{r_i(\bar \vu_{k, i})}{c_i(\bar \vu_{k,i})}} - 1 \Bigg| c_i(\bar \vu_{k, i}) \\
        \leq \;& \tau \sqrt{r_i(\bar \vu_{k, i}) c_i(\bar \vu_{k,i})} + \Big|\sqrt{r_i(\bar \vu_{k, i})} - \sqrt{c_i(\bar \vu_{k,i})} \Big| \sqrt{c_i(\bar \vu_{k,i})}, \\
        \big|e^{\bar \vu_{k + 1}\bl{i} - \bar \vu_{k}\bl{i}} - 1\big| r_i(\bar \vu_{k, i}) \leq \;& \Bigg|e^{\bar \vu_{k + 1}\bl{i} - \bar \vu_{k}\bl{i}} - \sqrt{\frac{c_i(\bar \vu_{k,i})}{r_i(\bar \vu_{k, i})}} \Bigg| r_i(\bar \vu_{k, i}) + \Bigg| \sqrt{\frac{c_i(\bar \vu_{k,i})}{r_i(\bar \vu_{k, i})}} - 1 \Bigg| r_i(\bar \vu_{k, i}) \\
        \leq \;& \tau \sqrt{r_i(\bar \vu_{k, i}) c_i(\bar \vu_{k,i})} + \Big|\sqrt{r_i(\bar \vu_{k, i})} - \sqrt{c_i(\bar \vu_{k,i})} \Big| \sqrt{r_i(\bar \vu_{k, i})}.
    \end{align*}
    Combining the above two displays completes the proof:
    \begin{align*}
        \|\nabla \varphi(\vu_{k+1})\|_1 \leq \; & \sum_{i=2}^n \Big|\sqrt{r_i(\vu_{k, i})} - \sqrt{c_i(\vu_{k, i})}\Big| \Big( \sqrt{r_i(\vu_{k, i})} + \sqrt{c_i(\vu_{k, i})}\Big) + 4\tau \sum_{i=1}^n \sqrt{r_i(\bar \vu_{k, i})c_i(\bar \vu_{k, i})} \\
        = \; &\sum_{i=2}^n |{r_i(\vu_{k, i})} - {c_i(\vu_{k, i})}| + 4 \tau \sum_{i=1}^n \sqrt{r_i(\bar \vu_{k, i})c_i(\bar \vu_{k, i})}\,.
    \end{align*}
\end{proof}

Next we provide an inexact version of Proposition~\ref{prop:grad-bound}. The original proof in \S\ref{sec:analysis} required bounding several types of terms in two different cases and the inexact version has two more such terms, so for the convenience of the reader, we collect these bounds in the following lemma. 

\begin{lemma}[Helper lemma for inexact analysis]\label{lem:helper-err}
  Fix any cycle $k \geq 0$ and let $\alpha > 0$ be an analysis parameter.
   Consider the following two cases of row/column imbalance for each $j \in [n]$:
    \begin{enumerate}[label=\Circled{\arabic*},font=\sffamily]
        \item $\big|\sqrt{r_j(\bar \vu_{k, j})} - \sqrt{c_j(\bar \vu_{k, j})}\big| < \alpha \big(\sqrt{r_j(\bar \vu_{k, j})} + \sqrt{c_j(\bar \vu_{k, j})}\big),$ 
        \item $\big|\sqrt{r_j(\bar \vu_{k, j})} - \sqrt{c_j(\bar \vu_{k, j})}\big| \geq \alpha \big(\sqrt{r_j(\bar \vu_{k, j})} + \sqrt{c_j(\bar \vu_{k, j})}\big).$ 
    \end{enumerate}
    For shorthand, define
    \begin{align*}
        \Xone :&= \sum_{j \, : \, \Circled{1}} \Big(\sqrt{r_j(\bar \vu_{k, j})} - \sqrt{c_j(\bar \vu_{k, j})} \, \Big)^2\,, \\
        \Yone :&= \sum_{j \, : \, \Circled{1}} \Big|r_j(\bar \vu_{k, j}) - c_j(\bar \vu_{k, j})\Big|\,, \\
        \Zone :&= \sum_{j \, : \, \Circled{1}} \sqrt{r_j(\bar \vu_{k, j})  c_j(\bar \vu_{k, j})}\,, \\
        \Sone :&= \sum_{j \, : \, \Circled{1}} \Big( r_j(\bar \vu_{k, j}) + c_j(\bar \vu_{k, j}) \Big)\,,
    \end{align*}
    and similarly for $\Xtwo$, $\Ytwo$, $\Ztwo$, and $\Stwo$. Then
    \begin{align*}
        \Yone &\leq 2 \alpha \Sone\,, \\
        \Zone &\leq \frac{1}{2} \Sone\,, \\
        \Ytwo &\leq \frac{1}{\alpha} \Xtwo\,, \\
        \Ztwo &\leq \frac{1}{4\alpha^2} \Xtwo\,, \\
        \Sone &\leq 2 \Big[ \varphi(\bar \vu_{k}) + (1+\tau)\frac{1+\alpha}{1-\alpha} \varphi(\bar \vu_{k+1}) \Big]\,.
    \end{align*}  
\end{lemma}
\begin{proof}
    For shorthand, denote $r_j = r_j(\bar \vu_{k, j})$ and $c_j = c_j(\bar \vu_{k, j})$.
    For the bound on $\Yone$, note that $|r_j - c_j| = |\sqrt{r_j} - \sqrt{c_j}| |\sqrt{r_j} + \sqrt{c_j}| 
    \leq \alpha (\sqrt{r_j} + \sqrt{c_j})^2 \leq 2\alpha (r_j + c_j)$ by $\Circled{1}$. For the bound on $\Zone$, note that $\sqrt{r_jc_j} \leq \tfrac{1}{2}(r_j+c_j)$ by the AM-GM inequality. For the bound on $\Ytwo$, note that $|r_j - c_j| \leq \tfrac{1}{\alpha} (\sqrt{r_j} - \sqrt{c_j})^2$ by $\Circled{2}$. For the bound on $\Ztwo$, note that $\sqrt{r_jc_j} \leq \tfrac{1}{4} (\sqrt{r_j} + \sqrt{c_j})^2 \leq \tfrac{1}{4\alpha^2}(\sqrt{r_j} - \sqrt{c_j})^2$ by $\Circled{2}$. 
    
    For the bound on $\Sone$, use the following inexact version of~\eqref{eq:inter-final-bound}:
    \begin{equation*}
    \begin{aligned}
        &r_j(\bar \vu_{k, j})
        \leq
        e^{\bar \vu_{k }\bl{j} - \bar \vu_{k + 1}\bl{j}} \sqrt{\frac{r_j(\bar \vu_{k, j})}{c_j(\bar \vu_{k, j})}}
        r_j(\bar \vu_{k + 1}) + r_j(\bar \vu_{k}) 
        \leq 
        (1 + \tau)\frac{1 + \alpha}{1 - \alpha}
        r_j(\bar \vu_{k + 1}) + r_j(\bar \vu_{k}) \,,
        \\
        &c_j(\bar \vu_{k, j})
        \leq 
        e^{\bar \vu_{k + 1}\bl{j} - \bar \vu_{k}\bl{j}}\sqrt{\frac{c_j(\bar \vu_{k, j})}{r_j(\bar \vu_{k, j})}}
        c_j(\bar \vu_{k+1}) + c_j(\bar \vu_{k})
        \leq (1 + \tau)\frac{1 + \alpha}{1 - \alpha}
        c_j(\bar \vu_{k+1}) + c_j(\bar \vu_{k})\,,
    \end{aligned}
    \end{equation*}
    and sum over all indices $j$, in the same way as in~\eqref{eq:grad-1}.
\end{proof}

Using this, we argue that if a cycle of the inexact Osborne algorithm does not improve the objective $\varphi$, then the current iterate must already be an $\cO(\epsilon)$-balancing.

\begin{lemma}[Convergence if no descent]\label{lem:no-descent-converge}
    Suppose $\eps \leq 1/2$ and $\tau \leq \epsilon^2$. If $\varphi(\bar \vu_{k + 1}) - \varphi(\bar \vu_k) \geq 0$ for some cycle $k$, then $\|\nabla \varphi(\bar \vu_{k + 1})\|_1 / \varphi(\bar \vu_{k + 1}) \leq 120\epsilon$. 
\end{lemma}
\begin{proof}
    We use the shorthand notation in Lemma~\ref{lem:helper-err} with $\alpha = \epsilon$. By using in order the no-descent assumption, the inexact descent lemma (Lemma~\ref{lem:descent-err}), and then Lemma~\ref{lem:helper-err} and the trivial bound $-\Xone \leq 0$ and $\tau \leq \epsilon^2$,
    \begin{align*}
        0 
        \leq 
        \varphi(\bar \vu_{k + 1}) - \varphi(\bar \vu_k)
        \leq 
        -\Xone + 2\tau \Zone - \Xtwo + 2\tau \Ztwo
        \leq 2\tau \Zone  -\frac{\Xtwo}{2}\,.
    \end{align*}
    Re-arranging and using $\tau \leq \eps^2$ gives $\Xtwo \leq 4\epsilon^2 \Zone$. By using the inexact imbalance lemma (Lemma~\ref{lem:imbalance}), this bound on $\Xtwo$, Lemma~\ref{lem:helper-err}, and crudely bounding constants, we obtain the desired result:
    \begin{align*}
         \|\nabla \varphi(\bar \vu_{k+1})\|_1
         \leq
         \Yone + 4\tau \Zone + \Ytwo + 4\tau \Ztwo
         \leq 12\epsilon \Sone
         \leq 120 \epsilon \varphi( \vu_{k+1})\,.
    \end{align*}
\end{proof}

Since Lemma~\ref{lem:no-descent-converge} shows that an $\cO(\epsilon)$-balancing is produced whenever a cycle of the inexact Osborne algorithm does not improve the objective, we can focus on the case that the objective is monotonically decreasing. In this setting, we prove the following inexact version of Proposition~\ref{prop:grad-bound}, which we recall was the key penultimate step from which our final runtime follows in \S\ref{sec:analysis}.

\begin{proposition}[Inexact version of Proposition~\ref{prop:grad-bound}]\label{prop:grad-bound-truncated}
Suppose $\ebar \leq 1/2$ and $\tau = \eps'^2/400$. For any initialization $\bar \vu_0 \in \R^n$ (not necessarily zero) and any number of cycles $K' \geq 1$ during which the potential $\varphi$ monotonically decreases, there exists $t \in [K']$ such that
\begin{align*}
    \frac{\|\nabla \varphi(\bar \vu_{t})\|_1}{\varphi(\bar \vu_{t})} \leq \;& 
    \ebar + \frac{50 \varphi(\bar \vu_0)}{ \ebar K' \varphi(\bar \vu_{K'})}.
\end{align*}
\end{proposition}
\begin{proof}
We largely follow the proof of Proposition~\ref{prop:grad-bound}, with additional care to bound the accumulated effect of the inexact updates. First, by the same averaging argument, there exists $k \in \{0, 1, \dots, K' - 1\}$ satisfying 
\begin{align}\label{eq:descent-k-err}
    \varphi(\bar \vu_{k}) - \varphi(\bar \vu_{k + 1}) \leq \frac{\varphi(\bar \vu_0) - \varphi(\bar \vu_{K'})}{K'}.
\end{align}
We will show that the claimed bound holds for $t = k+1$. For this $k$, let us use the shorthand notation in Lemma~\ref{lem:helper-err} with $\alpha = \ebar/20$. Set $\tau \leq \alpha^2$. By using in order the inexact descent lemma (Lemma~\ref{lem:descent-err}), and then Lemma~\ref{lem:helper-err} and the trivial bound $-\Xone \leq 0$,
    \begin{align*}
        \varphi(\bar \vu_{k + 1}) - \varphi(\bar \vu_k)
        \leq 
        -\Xone + 2\tau \Zone - \Xtwo + 2\tau \Ztwo
        \leq \alpha^2 \Sone - \frac{\Xtwo}{2} \,.
    \end{align*}
    Re-arranging gives $
    \Xtwo \leq 2\tau \Sone + 2 \big( \varphi(\bar \vu_{k}) - \varphi(\bar \vu_{k+1}) \big)$.
    By using the inexact imbalance lemma (Lemma~\ref{lem:imbalance}), this bound on $\Xtwo$, Lemma~\ref{lem:helper-err}, and then crudely bounding constants,
    \begin{align*}
         \|\nabla \varphi(\bar \vu_{k+1})\|_1
         &\leq
         \Yone + 4\tau \Zone + \Ytwo + 4\tau \Ztwo
         \\ &\leq 4\alpha(1+\alpha) \Bigg(  \Sone + \frac{\varphi(\bar \vu_{k}) - \varphi(\bar \vu_{k+1})}{2\alpha^2} \Bigg)
         \\ & \leq  \ebar \varphi(\bar \vu_{k+1}) +  \frac{50}{\ebar} \Big(\varphi(\bar \vu_{k}) - \varphi(\bar \vu_{k+1})\Big) 
         \,.
    \end{align*}
    The proof is complete by plugging in~\eqref{eq:descent-k-err}, dividing both sides by $\varphi(\bar \vu_{k+1})$, using the monotonicity assumption to bound $\varphi(\bar \vu_{k+1}) \geq \varphi(\bar \vu_{K'})$, and noting that by the definition of $\varphi$ we have $\varphi(\vu_{K'}) > 0$.
\end{proof}

Finally, we can now establish the claimed bit complexity result. For the convenience of the reader, we restate the result before proving it.

\bitComplexity*
\begin{proof}
    \textbf{Bit-complexity.} For $(\romannumeral1)$, $(\romannumeral2)$, $(\romannumeral3)$, this follows by the same argument as in~\cite[Theorem 13]{altschuler2023near} and choosing $\gamma = \Theta(\epsilon / n)$ and $\gamma' = \Theta(\epsilon^2)$ in Implementation~\ref{alg:implementation}. For $(\romannumeral4)$, apply Lemma~\ref{lem:termination-err} with $\bar \epsilon = \Theta(\epsilon)$. \textbf{Runtime.} If Implementation~\ref{alg:implementation} of cyclic Osborne does not decrease the potential value in some cycle, Lemma~\ref{lem:no-descent-converge} implies that the current iterate is $\Theta(\epsilon)$-balanced, hence we can successfully terminate. Thus, it suffices to consider the case that the algorithm monotonically decreases the potential value in each cycle. In this case, by following the proof of Theorem~\ref{thm:main} (with the usage of Proposition~\ref{prop:grad-bound} replaced by its inexact version in Proposition~\ref{prop:grad-bound-truncated}), we conclude that the number of arithmetic operations is at most $\cO\big(\frac{m\log \kappa}{\epsilon}\min\{\frac{1}{\epsilon}, \d\}\big)$.
\end{proof}

\end{document}